\documentclass{amsart}

\usepackage[latin9]{inputenc}
\usepackage{amsthm}
\usepackage{amstext}
\usepackage{amssymb}
\usepackage{enumitem}
\usepackage{hyperref}

\usepackage{amssymb}
\usepackage{amsmath}

\theoremstyle{plain}
\newtheorem{theorem}{Theorem}
  \newtheorem{lemma}[theorem]{Lemma}
    \newtheorem{proposition}[theorem]{Proposition}
  \theoremstyle{remark}
  \newtheorem{remark}[theorem]{Remark}
  \theoremstyle{definition}
    \newtheorem{definition}[theorem]{Definition}
        \newtheorem{example}[theorem]{Example}

\begin{document}
\title[Recent Progress in Shearlet Theory]{Recent Progress in Shearlet Theory: Systematic Construction of Shearlet Dilation Groups, Characterization of Wavefront Sets, and New Embeddings}

\author[G.\ S. Alberti]{Giovanni S. Alberti}
\address{Department of Mathematics,
ETH Z\"{u}rich, R\"{a}mistrasse 101, 8092 Z\"{u}rich, Switzerland}
\email{giovanni.alberti@sam.math.ethz.ch}


\author[S. Dahlke]{Stephan Dahlke}
\address{FB12 Mathematik und Informatik, Philipps-Universit\"at Marburg, Hans-Meerwein-Stra{\ss}e, Lahnberge, 35032 Marburg, Germany}
\email{dahlke@mathematik.uni-marburg.de}


\author[F. De Mari]{Filippo De Mari}
\address{Dipartimento di Matematica, Universit\`{a} di Genova, Via Dodecaneso 35, Genova, Italy}
\email{demari@dima.unige.it}

\author[E. De Vito]{Ernesto De Vito}
\address{Dipartimento di Matematica, Universit\`{a} di Genova, Via Dodecaneso 35, Genova, Italy}
\email{devito@dima.unige.it}

 
 \author[H. F\"{u}hr]{Hartmut F\"{u}hr}
 \address{Lehrstuhl A f\"{u}r Mathematik, RWTH Aachen University, 52056 Aachen, Germany}\email{fuehr@matha.rwth-aachen.de}
 
 \date{May 9, 2016}
 
 \subjclass[2010]{42C15, 42C40, 46F12, 22D10}

\maketitle
\begin{abstract}
The class of generalized shearlet dilation groups has recently been developed to allow the unified treatment of various shearlet groups and associated shearlet transforms that had previously been studied on a case-by-case basis. We consider several aspects of these groups: First, their systematic construction from associative algebras, secondly, their suitability for the characterization of wavefront sets, and finally, the question of constructing embeddings into the symplectic group in a way that intertwines the quasi-regular representation with the metaplectic one. For all questions, it is possible to treat the full class of generalized shearlet groups in a comprehensive and unified way, thus generalizing known results to an infinity of new cases. Our presentation emphasizes the interplay between the algebraic structure underlying the construction of the shearlet dilation groups, the geometric properties of the dual action, and the analytic properties of the associated shearlet transforms.
\end{abstract}

\section{Introduction}
\label{sec:1}
{This chapter is concerned with several important aspects of modern signal analysis. Usually, signals are modeled as elements of function spaces such as $L^2$ or Sobolev spaces. To analyze such a signal and to extract the information of interest from it, the first step is always to decompose the signal into suitable building blocks. This is performed by  transformation, i.e., the signal is mapped into function spaces on an underlying parameter set, and then the signal is processed and analyzed by studying and modifying the resulting coefficients. By now, a whole zoo of suitable transforms have been introduced and analyzed such as the Fourier transform, the Gabor transform, or the wavelet transform, just to name a few. Some of them have  already been very succesful, e.g., the Fourier transform works excellently for signals that are well-localized in the frequency domain, whereas wavelets are often the method of choice for the analysis of piecewise smooth signals with well-localized singularities such as edges 
in an image. Which transform to choose obviously depends on the application, i.e., on the type of information one wants to detect from the signal. However, in recent years, it has turned out that a serious bottleneck still has to be removed. Most of the classical transforms such as the wavelet transform perform suboptimally when it comes to the detection of directional information. The reason is very simple: most of these transforms are essentially isotropic, whereas directional information is of anisotropic nature. This observation triggered many innovative studies how to design new building blocks that are particularly tuned to this problem, such as curvelets \cite{CaDo2}, contourlets \cite{DoVe}, ridgelets \cite{CaDo1} and many others.
In this chapter, we are in particular interested in one specific contribution to this problem, i.e., the shearlet approach. Shearlets are new affine representation systems that are based on translations, shearings, and anisotropic dilations. We refer to the monograph \cite{shearlet_book} for an overview. Among all the new approaches, the shearlet transform stands out for the following reason: the continuous shearlet transform can be derived from a square-integrable representation of a specific group, the full shearlet group \cite{DKMSST, DaKuStTe, DaStTe10}. This property is not only of academic interest but has the important consequence that the whole powerful machinery derived in the realm of square-integrable group representations such as reproducing kernels, inversion formulas etc. can directly be employed. This feature of the shearlet transform clearly has strengthened the interest in the shearlet theory, and many important results concerning the group-theoretical background have been derived so far. It 
is the aim of this chapter to push forward, to clarify and to unify this theory with respect to several important aspects. Our main objectives can be described as follows.}

{After the full shearlet group has been discovered, the question arose if other suitable concepts of shearlet groups could be constructed. A first example was the shearlet Toeplitz group \cite{DaTe10}, where the shearing part of the group has a Toeplitz structure. As we will see later in Subsection \ref{subsect:ex} of this chapter, the full shearlet group and the shearlet Toeplitz group are in a certain sense the ``extreme'' cases of a general construction principle. In this view, the full shearlet group corresponds to the nilpotency class $n=2$, whereas the Toeplitz case corresponds to the nilpotency class $n=d$, where $d$ denotes the space dimension. Therefore, one would conjecture that there should be a lot of examples ``in between''. Indeed, in \cite{Fu_RT}, a positive answer has been given, and a first classification of low-dimensional shearlet groups has been derived. 
In this chapter, we further extend these results and present an approach to the systematic 
construction of suitable shearlet groups. The 
starting point is a general class of shearlet groups introduced in \cite{Fu_RT}. We say that a dilation group $H$ is a shearlet group if every $h \in H$ can be written as $h=\pm ds, d \in D, s \in S$ where D is a diagonal {\em scaling subgroup} and $S$ denotes a connected, closed abelian matrix group, the {\em shearing subgroup}. The key to understanding and constructing shearing subgroups lies in the realization that their associated Lie algebras carry a very useful associative structure. This associative structure also greatly facilitates the task of identifying the diagonal scaling groups compatible with a given shearing subgroup. Through the notion of Jordan-H\"{o}lder bases the problem of characterizing all suitable scaling group generators translates to a rather simple linear system of equations. It turns out that all examples known so far are special cases of this general construction.}

{In recent studies, it has also been observed that shearlets provide a very powerful tool in microlocal analysis \cite{grohs}, e.g., to determine the local regularity of a function. In the one-dimensional case, pointwise smoothness can very efficiently be detected by examining the decay of the continuous wavelet transform as the scale parameter $a$ tends to zero \cite{JaffardMeyer}. In the multivariate setting, pointwise smoothness does not cover all the geometic information one might be interested in. E.g., if the function under consideration exhibits singularities, one usually wants to know in which direction the function is singular. This can be described by the so-called {\em wavefront set} of a distribution. It has turned out that the continuous shearlet transform can be employed to detect this wavefront set, once again by studying its decay as the scaling parameter tends to zero. This property has been first observed in \cite{KuLa}, we also refer to \cite{grohs} for an overview. In \cite{FeFuVo}, this 
concept has been generalized to much more general classes of dilation groups. It has been shown that under natural assumptions, a wavefront set can again be detected by the decay of the voice transform. Essentially, two fundamental conditions are needed, that are related with the {\em dual action} of the dilation group $H$: the dual action must be microlocally admissible in direction $\xi$ and it must satisfy the $V$-cone approximation property at $\xi$, see Section \ref{subseq:crit_wave} for the precise definitions. If these properties hold for one point $\xi_0$ in the open dual orbit, a characterization of wavefront sets is possible. In this chapter, we show that both properties are satisfied for our general construction of shearlet dilation groups, provided that the infinitesimal generator $Y$ of the scaling subgroup satisfies $Y=\mbox{diag}(1, \lambda_2, \ldots, \lambda_d), 0 < \lambda_i<1, 2 \leq i \leq d$. Consequently, characterizations of wavefront sets are possible for a huge subclass of our general 
construction. It is worth mentioning that anisotropic dilations are necessary for the detection of wavefront sets, in particular the classical (isotropic) continuous wavelet transform would not do the job.}

{A third important issue we will be concerned with in this chapter is the relations of our general shearlet groups to other classical groups, in particular to the symplectic groups $Sp(d, \Bbb{R})$. The symplectic groups are one of the most important classical groups, because they play a prominent role in classical mechanics. We therefore investigate to which extent our shearlet dilation groups can be embedded into symplectic groups, in a way that intertwines the quasi-regular representation with the metaplectic representation. For the full shearlet groups and the shearlet Toeplitz groups, these issues have been studied in \cite{2014-differentfaces}, see also \cite{emily}. Their connected components can indeed be embedded into the symplectic groups, which yields  group isomorphisms of the positive parts of shearlet groups with the so--called TDS(d) subgroups that have already been studied in \cite{2013-demari-devito}. In this chapter, we generalize this result to dilation groups of the form  $G=\Bbb{R}^d \rtimes H$, 
where $H$ is a subgroup of $T(d,\Bbb{R})_{+}=\{h\in {\rm GL}(d,\Bbb{R}): \text{$h_{1,1}>0$ and $h_{i,j}=0$ for every $i>j$}\}$. We show that for any such group there exists a group embedding $\phi: G \rightarrow Sp(d, \Bbb{R})$, and that its quasi-regular representation is unitarily equivalent to $\mu \circ \phi$, where $\mu$ denotes the metaplectic representation of $Sp(d, \Bbb {R})$. Since the positive part of any shearlet group falls into this general category, the desired embeddings for shearlet groups follow from this result. Let us also mention the following very interesting fact: for the full shearlet dilation groups, such embeddings are never possible. Indeed, in \cite{2014-differentfaces} it has been shown that already for the two-dimensional full shearlet group there does not exist an injective continuous homomorphism into $Sp(2,\Bbb{R})$ or into any of its coverings.}

{Let us also mention a nice by-product of our construction. In recent studies \cite{DaHaTe12, DaKuStTe, DaStTe12, DaStTe10}, an interesting relation of the shearlet approach to the coorbit theory derived by Feichtinger and Gr\"{o}chenig \cite{FeiGr0, FeiGr1, FeiGr2, FeiGr3} has been established. Based on a square integrable group representation, coorbit space theory gives rise to canonical associated smoothness spaces, where smoothness is measured by the decay of the underlying voice transform. In \cite{DaHaTe12, DaKuStTe, DaStTe12, DaStTe10}, it has been shown that all the conditions needed in the coorbit setting can be satisfied for the full shearlet and the shearlet Toeplitz groups. In \cite{Fu_coorbit}, the coorbit approach has been extended to much more general classes of dilation groups, and it turns out that the analysis from \cite{Fu_coorbit} also carries over to the construction presented in this chapter, so that we obtain many new examples of coorbit spaces. In particular, we refer to \cite{Fu_RT} 
for explicit criteria for compactly supported functions that can serve as atoms in the coorbit scheme.}

{This chapter is organized as follows. In Sections \ref{sec:2} and \ref{sec:3_Consruct}, we present our construction of generalized shearlet dilation groups. After discussing the basic notations and definitions in the Subsections \ref{subsec:2.1} and \ref{sect:gen_shearlets}, in Subsection \ref{subsec:shear_Lie} we start with the systematic investigation of the Lie algebras  of shearing subgroups. One of the main results is Lemma \ref{lem:desc_sg} which provides a complete description of a shearing subgroup in terms of the canonical basis of its Lie algebra. This fact can be used to derive linear systems whose nonzero solutions determine the anisotropic scaling subgroups that are compatible with $S$ (Lemma \ref{lem:char_diag}). These relationships are then used in Section \ref{sec:3_Consruct} to derive a systematic construction principle. The canonical basis can be directly computed from the structure constants of a Jordan-H\"older basis (Lemma \ref{lem:const_admis}). The power of this approach is 
demonstrated by several examples. In Section \ref{sec:anis_wave}, we study the suitability of shearlet dilation groups for the characterization of wavefront sets. Here the main result is Theorem \ref{thm:wfset_general_shearlet} which shows that shearlet groups with anisotropic dilations and suitable infinitesimal generators for the scaling subgroups do the job. The proof is performed by verifying the basic conditions from \cite{FeFuVo}. The last section is concerned with the embeddings of shearing dilation groups into symplectic groups. The main result of this section is Theorem \ref{thm:embedding} which shows that the huge class of semidirect products of the form $G= \Bbb{R}^d \rtimes H$, where $H$ is a subgroup of $T(d,\Bbb{R})_{+}$ can be embedded into $Sp(d, \Bbb{R})$.}

\section{Generalities on shearlet dilation groups}
\label{sec:2}

\subsection{Basic notations and definitions}
\label{subsec:2.1}
This chapter is concerned with the construction and analysis of large classes of generalized shearlet transforms. These transforms are constructed by fixing a suitable matrix group, the so-called {\em shearlet dilation group}. By construction, these groups have a naturally associated isometric continuous wavelet transform, which will be the generalized shearlet transform. In this subsection, we summarize the necessary notation related to general continuous wavelet transforms in higher dimensions. We let ${\rm GL}(d,\Bbb{R})$ denote the group of invertible $d\times d$-matrices. We use $I_d$ to denote the $d \times d$ identity matrix. The Lie algebra of ${\rm GL}(d,\Bbb{R})$ is denoted by $\mathfrak{gl}(d,\Bbb{R})$, which is the space of all $d \times d$
matrices, endowed with the Lie bracket $[X,Y] = XY-YX$. Given $h\in \mathfrak{gl}(d,\Bbb{R})$ its (operator) norm is denoted by
\[
\|h\|=\sup_{ |x|\leq 1} | h x|.
\]

We let $\exp: \mathfrak{gl}(d,\Bbb{R}) \to {\rm GL}(d,\Bbb{R})$ denote the exponential map, defined by
 \[
 \exp(X) = \sum_{k=0}^\infty \frac{X^k}{k!} 
 \]
 known to converge absolutely for every matrix $X$. Given a closed subgroup \linebreak{$H< {\rm GL}(d,\Bbb{R})$}, the associated Lie algebra of $H$ is denoted by $\mathfrak{h}$, and it is defined as tangent space of $H$ at $I_d$, or, equivalently, as the set of all
 matrices $X$ with $\exp(\Bbb{R} X) \subset H$. It is a Lie subalgebra of $\mathfrak{gl}(d,\Bbb{R})$, i.e., it is closed under taking Lie brackets. 

A matrix group of particular importance for the following is the group $T(d,\Bbb{R})$ of upper triangular matrices with ones on the diagonal. Elements of $T(d,\Bbb{R})$ are called {\em unipotent}. Its Lie algebra is the subspace $\mathfrak{t}(d,\Bbb{R}) \subset \mathfrak{gl}(d,\Bbb{R})$ of all strictly upper triangular matrices. It is well-known that $\exp:\mathfrak{t}(d,\Bbb{R}) \to T(d,\Bbb{R})$ is a homeomorphism \cite{HiNe}. In particular, whenever $\mathfrak{s} \subset\mathfrak{t}(d,\Bbb{R})$ is a Lie subalgebra, the exponential image $\exp(\mathfrak{s})$ is a closed, simply connected and connected matrix group with Lie algebra given by $\mathfrak{s}$. Conversely, any connected Lie subgroup $S$ of $T(d,\Bbb{R})$ is closed, simply connected and $S=\exp(\mathfrak{s})$ where $\mathfrak{s}\subset \mathfrak{t}(d,\Bbb{R})$ is the corresponding Lie algebra, see Theorem~3.6.2 of \cite{raja}. 

For the definition of generalized wavelet transforms, we fix a closed matrix group $H < {\rm GL}(d,\Bbb{R})$, the so-called {\em dilation group}, and let $G =\Bbb{R}^d \rtimes H$. This is the group of affine mappings generated by $H$ and all translations. Elements of $G$ are denoted by pairs $(x,h) \in \Bbb{R}^d \times H$, and the product of two group elements is given by $(x,h)(y,g) = (x+hy,hg)$. The left Haar measure of $G$ is given by $d\mu_G(x,h) = |\det(h)|^{-1}dx dh$, where $dx$ and $dh$ are the Lebesgue measure and the (left) Haar measure of $\Bbb{R}^d$ and $H$, respectively.

The group $G$ acts unitarily on ${\rm L}^2(\Bbb{R}^d)$ by the {\em quasi-regular representation} defined by
\begin{equation} \label{eqn:def_quasireg}
[\pi(x,h) f](y) = |{\rm det}(h)|^{-1/2} f\left(h^{-1}(y-x)\right)~.
\end{equation}
We assume that $H$ is chosen {\em irreducibly admissible}, i.e. such that $\pi$ is an {\em (irreducible) square-integrable representation}. Recall that a representation is irreducible if the only invariant closed subspaces of the representation space are the trivial ones. Square-integrability of the representation means that there exists at least one nonzero {\em admissible vector} $\psi\in {\rm L}^2(\Bbb{R}^d)$ such that the matrix coefficient
\begin{eqnarray}
(x,h) \mapsto \langle \psi, \pi(x,h) \psi \rangle\nonumber
\end{eqnarray} is in ${\rm L}^2(G)$, which is the ${\rm L}^2$-space associated
to the left Haar measure $d\mu_G $. In this case the associated wavelet transform
\begin{equation}
W_\psi : {\rm L}^2(\Bbb{R}^d) \ni f \mapsto \left(
(x,h) \mapsto \langle f, \pi(x,h) \psi \rangle \right) \in L^2(G)
\end{equation} 
is a scalar multiple of an isometry, which gives rise to the {\em wavelet inversion formula}
\begin{equation} \label{eqn:wvlt_inv}
f = \frac{1}{c_\psi} \int_G W_\psi f(x,h) \pi(x,h) \psi ~
d\mu_G(x,h)~,
\end{equation}
 where the integral is in the weak sense.

We note that the definition of $W_\psi f$ also makes sense for tempered distributions $f$, as soon as the wavelet $\psi$ is chosen as a Schwartz function and the ${\rm L}^2$-scalar product is properly extended to a sesquilinear map $\mathcal{S}' \times \mathcal{S} \to \Bbb{C}$. Analogs of the wavelet inversion formula are not readily available in this general setting, but it will be seen below that the transform has its uses, for example in the characterization of wavefront sets.

Most relevant properties of the wavelet transform are in some way or another connected to the {\em dual action}, i.e., the (right) linear action $\Bbb{R}^d \times H \ni(\xi,h) \mapsto h^T \xi$. For example, $H$ is irreducibly admissible if and only if the dual action has a single open orbit $\mathcal{O} = \{ h^T \xi_0 : h \in H \} \subset \Bbb{R}^d$ of full measure (for some $\xi_0 \in \mathcal{O}$), such that in addition the stabilizer group $H_{\xi_0} = \{ h \in H : h^T \xi_0 = \xi_0 \}$ is compact \cite{Fu10}. This condition does of course not depend on the precise choice of $\xi_0 \in \mathcal{O}$. The dual action will also be of central importance to this chapter.

\subsection{Shearlet dilation groups}
\label{sect:gen_shearlets}

The original shearlet dilation group was introduced in \cite{DKMSST, DaKuStTe}, as
\[
 H = \left\{ \pm \left( \begin{array}{cc} a & b \\ 0 &  a^{1/2} 
 \end{array} \right) : a>0, b \in \Bbb{R} \right\}~.
\]
The rationale behind this choice was that the anisotropic scaling, as prescribed by the 
exponents $1, 1/2$ on the diagonal, combines with the shearing (controlled by the parameter $b$) to provide a system
of generalized wavelets that are able to swiftly adapt to edges of all orientations (except one). 
A mathematically rigourous formulation of this property is the result, due to Kutyniok and Labate, that the continuous shearlet transform
characterizes the wavefront set \cite{KuLa}. 
Approximation-theoretic properties of a different, more global kind were the subject of the chapter 
\cite{DaKuStTe} describing the so-called {\em coorbit spaces} defined in terms of
weighted integrability conditions on the wavelet coefficients. 

The original shearlet dilation group has since been generalized to higher dimensions. 
Here, the initial construction was introduced in \cite{DaStTe10}, and further studied, e.g., in \cite{DaStTe12,CzaKi12}. 
It is a matrix group in dimension $d \ge 3$ defined by 
\begin{equation}
\label{eqn:clsh_ddim}
S = \left\{ \pm \left( 
\begin{array}{cccc} a & s_1 &\ldots & s_{d-1} \\ 
  & a^{\lambda_2} &   &  \\
		&  & \ddots & \\
		&  & & a^{\lambda_d}  
		\end{array} \right) : a >0 ,~s_1,\ldots,s_{d-1} \in \Bbb{R} \right\}.
\end{equation}
Here $\lambda_2,\ldots,\lambda_d$ are positive exponents, often chosen as $\lambda_2 = \ldots = \lambda_d = 1/2$. It should, however, be noted that they can be chosen essentially arbitrarily (even negative), without affecting the wavelet inversion formula.  Coorbit space theory is applicable to all these groups as well \cite{DaStTe12,Fu_coorbit}. Furthermore, it was recently shown that the associated shearlet transform also characterizes the wavefront set \cite{FeFuVo}, as long as the exponents $\lambda_2,\ldots,\lambda_d$ are strictly between zero and one. 

A second, fundamentally different class of shearlet groups are the {\em Toeplitz shearlet groups} introduced in \cite{DaTe10} and further studied in \cite{DaHaTe12}. These groups are given by
\begin{equation} 
\label{eqn:tsh_ddim} 
H = \left\{ \pm \!\left( 
\begin{array}{ccccccc} a & \hspace*{2mm} s_1 & \hspace*{2mm}s_2 & \hspace*{2mm} \ldots & \hspace*{2mm}\ldots & \hspace*{2mm}\ldots & \hspace*{2mm} s_{d-1} \\ 
  & a & s_1 & s_2 &\ldots &  \ldots & s_{d-2} \\
		&  & \ddots & \ddots & \ddots &  & \vdots \\
                  &  & & \ddots & \ddots & \ddots & \vdots \\
		&  &  & &\ddots & \ddots & s_2\\  & & &  & & \ddots & s_1 \\
&  & & & & & a \
\end{array} \right) \!: a>0,~s_1,\ldots,s_{d-1} \in \Bbb{R} \right\}. 
\end{equation}
Coorbit space theory can be applied to these groups as well \cite{DaHaTe12,Fu_coorbit}. By \cite[Lemma 4.10]{FeFuVo}, the fact that $H$ contains nontrivial multiples of the identity implies that $H$ does not characterize the wavefront set. However, it will be shown below that by properly adjusting the diagonal entries, it is possible to construct a closely related group $H'$ that does lend itself to the characterization of the wavefront set. 

A closer inspection of the two higher-dimensional families of shearlet group reveals several  common traits: fix one of the above-listed groups $H$. Then each $h \in H$ factors as
\[
 h = \pm {\rm diag}(a,a^{\lambda_2},\ldots,a^{\lambda_d}) \cdot u 
\]where the first factor denotes the diagonal matrix with the same diagonal entries as $h$, and the second factor $u$ is unipotent. In fact, this factorization is necessarily unique. Furthermore, denoting by $D$ the set of all diagonal matrices occurring in such factorizations, and by $S$ the set of all unipotent ones that arise, it is easy to see that $D$ (and consequently $S$) are closed subgroups of $H$. Finally, one readily verifies that the groups $S$ that occur in the examples are in fact commutative. We will now use these properties to define a general class of shearlet dilation groups, that we will study in this chapter:

\begin{definition}
Let $H < {\rm GL}(d,\Bbb{R})$ denote an irreducibly admissible dilation group. $H$ is called {\em generalized shearlet dilation group}, if there exist two closed subgroups $S, D < H$ with the following properties:
\begin{enumerate}[label=(\roman*)]
\item $S$ is a connected abelian Lie subgroup of $T(d,\Bbb{R})$;
\item $D = \{ \exp(r Y) : r \in \Bbb{R} \}$ is a one-parameter group, where $Y$ is a diagonal matrix;
\item Every $h \in H$ can be written uniquely as $h = \pm d s$, with $d \in D$ and $s \in S$.
\end{enumerate}
$S$ is called the {\em shearing subgroup} of $H$, and $D$ is called the {\em diagonal complement} or {\em scaling subgroup} of $H$.
\end{definition}
\begin{remark}
As noted in Subsection \ref{subsec:2.1}, $S$ is closed, simply connected and the exponential map is a diffeomorphism from its Lie algebra $\mathfrak{s}$ onto $S$.
\end{remark}
 \begin{remark}
 The class of shearlet dilation groups was initially defined in \cite{Fu_RT}, and for some of the following results and observations, more detailed proofs can be found in that paper. In particular, it was shown there that coorbit space theory applies to all generalized shearlet dilation groups. In fact, it is possible to construct wavelet frames with  compactly supported atoms, with frame expansions that, depending on the provenance of the signal, converge in a variety of coorbit space norms simultaneously. 
 \end{remark}
 
As will be seen below, shearlet dilation groups can be constructed systematically. The natural order in finding the constituent subgroups $S,D$ is to first pick a candidate for $S$, and then determine the infinitesimal generators of the one-parameter group $D$ that are compatible with $S$. The details of this programme are given in the next subsections. 
 
\subsection{Shearlet dilation groups and their Lie algebras}
\label{subsec:shear_Lie}
It is the aim of this subsection to give an overview of the most important structural properties of shearlet dilation groups. 
The following proposition gives a first characterization of these groups, see \cite[Proposition 4.3]{Fu_RT}. 
\begin{proposition} \label{prop:char_shearsubs}
Let $S$ denote a connected abelian subgroup of $T(d,\Bbb{R})$.  Then the following are equivalent:
\begin{enumerate}[label=(\roman*)]
\item $S$ is the shearing subgroup of a generalized shearlet dilation group;
\item There is $\xi \in \Bbb{R}^d$ such that $S$ acts freely on $S^T \xi$ via the dual action, and in addition, ${\rm dim}(S) = d-1$;
\item The matrix group $A= \{ rs: s \in S,  r \in \Bbb{R}^\times \}$ is an abelian irreducibly admissible dilation group. It is also a shearlet dilation group. 
\end{enumerate}
\end{proposition}

The fundamental observation made in \cite[Remark 9]{Fu98} is that if $A$ is abelian and admissible, as in part (iii) of the above proposition, then its Lie algebra $\mathfrak{a}$ is in fact an {\em associative} subalgebra containing the identity element, hence it is closed under matrix multiplication. This associative structure is in many ways decisive. To begin with, one has the relations
\[
\mathfrak{a} = {\rm span}(A)~,~ A = \mathfrak{a}^{\times}
\] i.e., $A$ consists precisely of the multiplicatively invertible elements of the associative algebra $\mathfrak{a}$. We will see in Subsection \ref{subsect:const_sg} below that this connection to associative algebras can be used for the systematic --even exhaustive-- construction of shearing subgroups. 
 
There is however a second ingredient, that is more directly related to the properties of the dual action. It is described in the following lemma, see \cite[Corollary 4.7]{Fu_RT}. We use $e_1,\ldots,e_d$ for the canonical basis of $\Bbb{R}^d$.
\begin{lemma} \label{lem:char_admiss}
 Let $S$ denote a connected abelian subgroup of $T(d,\Bbb{R})$ of dimension $d-1$, with Lie algebra $\mathfrak{s}$.  Then the following are equivalent:
\begin{enumerate}[label=(\roman*)]
 \item $S$ is a shearing subgroup;
 \item There exists a unique basis $X_2,\ldots,X_d$ of $\mathfrak{s}$ with $X_i^T e_1 = e_i$, for all $i=2,\ldots,d$.
 \end{enumerate}
 We call the basis from part (ii) the {\em canonical basis} of $\mathfrak{s}$.
\end{lemma}
The canonical basis plays a special role for the description of shearing subgroups. As a first indication of its usefulness, we note that all off-diagonal entries of the elements of shearing groups depend linearly on the entries in the first row. 
\begin{lemma} \label{lem:desc_sg}
 Let $S$ denote a shearing subgroup with Lie algebra $\mathfrak{s}$, and canonical basis $X_2,\ldots,X_d$ of $\mathfrak{s}$. Then the following holds:
 \begin{enumerate}[label=(\alph*)]
  \item $S = \{ I_d + X : X \in \mathfrak{s} \}$.
  \item Let $h \in S$ be written as 
  \[
     h = \left( \begin{array}{cccccc} 1 & h_{1,2} & \ldots & \ldots & \ldots & h_{1,d} \\
   0 & 1 & h_{2,3} &  \ldots & \ldots & h_{2,d} \\
    0 & 0 & \ddots &  \ddots & \vdots & \vdots \\
    0 & 0 & 0 &  \ddots & \ddots & \vdots \\
    0 & 0 & 0 & 0 & 1 & h_{d-1,d} \\
    0 & 0 & 0 & 0 & 0 & 1 
             \end{array} \right).
 \] Then
 \[
  h = I_d + \sum_{i=2}^d h_{1,i} X_i~. 
 \]
 \end{enumerate}
\end{lemma}
\begin{proof}
 For part (a), denote the right-hand side by $S_1$. Since $\mathfrak{s}$ is an associative subalgebra consisting of nilpotent matrices, $S_1$ consists of invertible matrices, and it is closed under multiplication.
 Furthermore, the inverse of any element of $S_1$ can be computed by a Neumann series that breaks off after at most $d$ terms:
 \[
  (I_d + X)^{-1} = I_d  + \sum_{k=2}^{d-1} (-1)^{k-1} X^k~, 
 \]
and the result is again in $S_1$. Hence $S_1$ is a matrix group. It is
obviously closed and connected, with tangent space of
$S_1$ at the identity matrix given by $\mathfrak{s}$. It follows that $S_1$ is a Lie subgroup of $T(d,\Bbb R)$ and, hence, it is simply connected. Thus $S$ and $S_1$ are closed, connected, and simply connected subgroups sharing the same Lie algebra, hence they are
 equal. Now part (b) directly follows from (a) and the properties of the canonical basis. 
\end{proof}

We now turn to the question of characterizing the scaling subgroups $D$ that are compatible with a given shearing subgroup $S$.
 It is convenient to describe $D$ in terms of its Lie algebra as well. Since $D$ is one-dimensional, we have $D = \exp(\Bbb{R} Y)$, with a diagonal matrix $Y = {\rm diag}(\lambda_1,\lambda_2,\ldots,\lambda_d)$. We then have the following criterion \cite[Proposition 4.5]{Fu_RT}:
 \begin{proposition} \label{prop:char_Y}
Let $S< {\rm GL}(d,\Bbb{R})$ denote a shearing subgroup. Let $Y$ denote a nonzero diagonal matrix, and let $D := {\rm exp}(\Bbb{R}Y)$ the associated one-parameter group with infinitesimal generator $Y$. Then the following are equivalent:
\begin{enumerate}[label=(\roman*)]
\item $H = DS \cup (-DS)$ is a shearlet dilation group;
\item For all $X \in \mathfrak{s}$ we have $[X,Y] = XY-YX \in \mathfrak{s}$, and in addition the first diagonal entry of $Y$ is nonzero.
\end{enumerate}
\end{proposition}
{
\begin{remark}
The above proposition states that $H=S \rtimes \Bbb R^\times$, so
that $H$ is solvable group with two connected components, and each of them is
simply connected.  
\end{remark}
}
Since $Y$ and $rY$, for nonzero $r \in \Bbb{R}$, determine the same one-parameter subgroup, part (ii) of the proposition allows to fix $\lambda_1 = 1$. Note that part (ii) is trivially fulfilled by {\em isotropic scaling}, which corresponds to taking $1 = \lambda_1 = \lambda_2 = \ldots = \lambda_d$. In what follows, we will be particularly interested in anisotropic solutions; our interest in these groups is mainly prompted by the crucial role of anisotropic scaling for wavefront set characterization.  

It turns out that the relation $[Y, \mathfrak{s}] \subset \mathfrak{s}$ translates to a fairly transparent system of linear equations. Once again, the canonical basis $X_2,\ldots,X_d$ of $\mathfrak{s}$ proves to be particularly useful: As the following lemma shows, the adjoint action $\mathfrak{s} \ni X \mapsto [Y, X]$ maps $\mathfrak{s}$ into itself if and only if the $X_i$ are eigenvectors of that map.  The lemma uses the notation $E_{i,j}$ for the matrix having entry one at row $i$ and column $j$, and zeros everywhere else. 

\begin{lemma} \label{lem:char_diag}
Let $\mathfrak{s}$ denote the Lie algebra of a shearing subgroup, and let $X_2,\ldots, X_d$ denote the canonical basis of $\mathfrak{s}$, given by 
\begin{equation} \label{defn:struct_Xi_2}
  X_i = E_{1,i} +\sum_{j=2}^{d}  \sum_{k=j+1}^d d_{i,j,k} E_{j,k}~
 \end{equation} with suitable coefficients $d_{i,j,k}$.
 Let $Y = {\rm diag}(1,\lambda_2,\ldots,\lambda_d)$ be given. 
Then $[Y,\mathfrak{s}] \subset \mathfrak{s}$ if and only if
\[
  \mbox{for all } i=2,\ldots,d~:~ \lambda_i = 1 + \mu_i~,
\] and the vector $(\mu_2,\ldots,\mu_d)$ is a solution of the system of linear equations given by
\begin{equation} \label{eqn:char_diag}
 \mbox{for all } (i,j,k) \in \{ 2, \ldots, d \}^3 \mbox{ with } d_{i,j,k} \not= 0 ~:~ \mu_i + \mu_j = \mu_k~. 
 \end{equation}
 In particular, $(\mu_2,\ldots,\mu_d) \mapsto (1,1+\mu_2,\ldots,1+\mu_d)$ sets up a bijection between the nonzero solutions of (\ref{eqn:char_diag}) on the one hand and the anisotropic scaling subgroups $D$ compatible with $S$ on the other. 
\end{lemma}
\begin{remark} 
Note that (\ref{defn:struct_Xi_2}) shows that $d_{i,j,k}=(X_i)_{jk}$.
\end{remark}
\begin{proof}
 We first note that the $E_{j,k}$ are eigenvectors under the adjoint action of any diagonal matrix:
 \begin{equation} 
  [Y, E_{j,k}] = (\lambda_j - \lambda_k) E_{j,k}~. 
 \end{equation}
As a consequence, given any matrix $X$, the support of the matrix $[Y,X]$ (i.e., the set of indices of its nonzero entries) is contained in the support of $X$. 

Note that $Y$ normalizes $\mathfrak{s}$ if and only if $[Y,X_i] \in \mathfrak{s}$ for $i=2,\ldots,d$. Now the calculation  
\begin{equation} \label{eqn:comp_comm}
 [Y,X_i] =[Y,E_{1,i}] +  \sum_{(j,k)} d_{i,j,k} [Y,E_{j,k}] = (1-\lambda_i) E_{1,i} + \sum_{(j,k)} d_{i,j,k} (\lambda_{j}-\lambda_k) E_{j,k}
\end{equation} shows that the only (potentially) nonzero entry in the first row of $[Y,X_i]$ occurs at the $i$th column, hence $[Y,X_i]$ is in $\mathfrak{s}$ if and only if it is a scalar multiple of $X_i$. In view of (\ref{eqn:comp_comm}) and the linear independence of the $E_{j,k}$, this holds precisely when 
\begin{equation} \label{eqn:char_diag_2}
 \mbox{for all } (i,j,k)  \in \{ 2, \ldots, d \}^3 \mbox{ with } d_{i,j,k} \not= 0 ~:~ 1-\lambda_i = \lambda_j - \lambda_k~. 
 \end{equation}
Rewriting this system for $\mu_i = \lambda_i - 1$, for $i=2,\ldots,d$, yields (\ref{eqn:char_diag}).
\end{proof}

Finally, let us return to properties of the associated shearlet transforms. 
In view of the central role of the dual action, it is important to compute the associated open dual orbit. Here we have the following, see \cite[Proposition 4.5]{Fu_RT}:
\begin{proposition} \label{prop:open_orbit}
 Let $S$ be a shearing subgroup, and $D$ any diagonal complement of $S$. Then $H = DS \cup  -DS$ acts freely on the unique open dual orbit given by $\mathcal{O} = \Bbb{R}^\times \times \Bbb{R}^{d-1}$. 
\end{proposition}

Note that the dual orbit is the same for all shearing groups. Somewhat surprisingly, the same can be said of the admissibility condition \cite[Theorem 4.12]{Fu_RT}:
\begin{theorem} \label{thm: shear_adm_cond}
Let $H < {\rm GL}(\Bbb{R}^d)$ denote a generalized shearlet dilation group. Then $\psi \in {\rm L}^2(\Bbb{R}^d)$ is admissible iff
\[
 \int_{\Bbb{R}^d} \frac{|\widehat{\psi}(\xi)|^2}{|\xi_1|^d} d \xi < \infty~. 
\]
\end{theorem}

\section{A construction method for shearlet dilation groups}
\label{sec:3_Consruct}
\subsection{Constructing shearing subgroups}
\label{subsect:const_sg}
In this subsection we want to describe a general method for the systematic construction of shearing subgroups. Recall that given a shearing subgroup $S$ with Lie algebra $\mathfrak{s}$, taking the Lie algebra $\mathfrak{a} = \Bbb{R} I_d \oplus \mathfrak{s}$ and its associated closed matrix group $A$ results in an abelian irreducibly admissible matrix group. Following \cite{Fu98}, this entails that $\mathfrak{a}$ is an associative matrix algebra. Furthermore, note that $\mathfrak{s}$ consists of strictly upper triangular matrices, which entails that any product of $d$ elements of $\mathfrak{s}$ vanishes. 

These features of $\mathfrak{a}$ can be described in general algebraic terms. Given a finite-dimensional, associative commutative algebra $\mathcal{A}$, we call an element $a \in \mathcal{A}$ {\em nilpotent} if there exists $n \in \Bbb{N}$ such that $a^n = 0$. The set of all nilpotent elements in $\mathcal{A}$ is called the {\em nilradical} of $\mathcal{A}$, denoted by $\mathcal{N}$. We call $\mathcal{A}$ nilpotent if every element of $\mathcal{A}$ is nilpotent. $\mathcal{N}$ is
an {\em ideal} in $\mathcal{A}$, i.e., given $a \in \mathcal{N}$ and an arbitrary $b \in \mathcal{A}$, one has $(ab)^n = a^n b^n =0$ for sufficiently large $n$, i.e. $ab$ is again in the nilradical.
We call the algebra $\mathcal{A}$ {\em irreducible (over $\Bbb{R}$)} if it has a unit element $1_{\mathcal{A}}$ satisfying $1_{\mathcal{A}} b = b$ for all $b \in \mathcal{A}$, and such that $\mathcal{A} = \Bbb{R} \cdot 1_{\mathcal{A}} \oplus \mathcal{N}$ holds. Note that $\mathcal{N}$ determines $\mathcal{A}$ in this case, and we will freely switch between $\mathcal{A}$ and $\mathcal{N}$ in the following.

Now the above considerations show that $\mathfrak{a}$ is an irreducible associative commutative algebra. In the remainder of this subsection, we will be concerned with a converse to this statement, i.e., with the construction of shearing subgroups from an abstractly given irreducible associative algebra. 
 Assume that $\mathcal{A}$ is an irreducible commutative associative algebra of dimension $d$, and denote  its nilradical by $\mathcal{N}$. 
 We let 
\[
 n(\mathcal{A}) = \min \{ k \in \Bbb{N}:  a^k = 0, \; \forall a \in \mathcal{N} \}~,
\] which is called the {\em nilpotency class} of $\mathcal{A}$. 
 Letting 
 \[
  \mathcal{N}^k = \{ a_1 \ldots a_k : a_i \in \mathcal{N} \}, 
 \] for $k\ge 1$, and $\mathcal{N}^{0} = \mathcal{N}$, one can prove that 
 \[
  n(\mathcal{A}) = \min \{ k \in \Bbb{N}:  \mathcal{N}^k = \{ 0 \} \} \le d~.
 \]
By definition of the nilpotency class, we obtain that $\mathcal{N}^{n(\mathcal{A})-1} \not= \{ 0 \}$, and for all $a \in \mathcal{N}^{n(\mathcal{A})-1}$ and $b \in \mathcal{N}$, it follows that $ab = 0$. 

Hence, choosing a nonzero $a_d \in \mathcal{N}^{n(\mathcal{A})-1} $, we find that $\mathcal{I}_d := \Bbb{R} \cdot a_d$ is an ideal in $\mathcal{N}$; in fact, we get $\mathcal{N} \mathcal{I}_d = \{ 0 \}$. Applying the same reasoning to the algebra $\mathcal{N}/\mathcal{I}_d$ (and choosing any representative modulo $\mathcal{I}_d$) produces a second element $a_{d-1}$ with the property that $\mathcal{I}_{d-1} = {\rm span}(a_{d-1},a_d)$ fulfills $\mathcal{N} \mathcal{I}_{d-1} \subset \mathcal{I}_d$. Further repetitions of this argument finally yield a basis
$a_2,\ldots, a_d$ of $\mathcal{N}$, that we supplement by 
$
 a_1 = 1_{\mathcal{A}} 
$ to obtain a basis of $\mathcal{A}$ with the property 
\begin{equation} \label{eqn:ideals_proper}
 \mathcal{N} \mathcal{I}_k \subset \mathcal{I}_{k+1} \mbox{ for } 1 \le k < d~,
\end{equation} and $\mathcal{I}_2 = \mathcal{N}$. We call a basis
$a_2,\ldots,a_d$ of $\mathcal{N}$ satisfying condition
(\ref{eqn:ideals_proper}) a {\em Jordan-H\"older basis} of
$\mathcal{N}$. 

The existence of a Jordan-H\"older basis can be also proved by referring to a general result about nilpotent representations of nilpotent algebras, see Theorem 3.5.3 of \cite{raja}. Indeed, regard $\mathcal N$ as nilpotent algebra and $\mathcal A$ as a vector space. It is easy to check that  the (regular) representation $\rho$ of the Lie algebra  $\mathcal N$ acting on $\mathcal A$ as $\rho(a)b=a b$ is nilpotent, so that there exists a basis $\{a_1,\ldots,a_d\}$ of $\mathcal A$ such that for each $a\in\mathcal N$, the endomorphism $\rho(a)$ is represented by a strictly upper triangular matrix $\Psi(a)\in \mathfrak{gl}(d,\Bbb{R})$ according to the canonical isomorphism 
\[
\rho(a) a_j=\sum_{k=1}^d \Psi(a)_{j,k}\, a_k \qquad j=1,\ldots,d\ .
\]
Since $\rho(a)1_{\mathcal A}=a$, it is always possible to choose $a_1=1_{\mathcal A}$ and, by construction, for all $a\in\mathcal N$ and for $i=1,\ldots,d-1$ 
\[\rho(a){\rm span}\{a_{i},\ldots,a_d\}\subset
{\rm span}\{a_{i+1},\ldots,a_d\} \qquad \rho(a)a_d=0\ . \] 

These bases provide accesss to an explicit construction of an associated shearing subgroup, explained in detail in the next lemma. Recall the notation $E_{i,j}$ for the matrix possessing entry one in row $i$, column $j$, and zeros elsewhere. Note that the map $\Psi: \mathcal{A} \to \mathfrak{gl}(d,\Bbb{R})$ in the following lemma coincides with the identically denoted map that we just introduced. 

\begin{lemma} \label{lem:const_admis}
 Let $\mathcal{A}$ denote an irreducible commutative associative algebra of dimension $d$ with nilradical $\mathcal{N}$ possessing the Jordan-H\"older basis 
 \[
a_2,\ldots, a_d \in \mathcal{A} ~.
 \] Let $a_1 = 1_{\mathcal{A}}$, and let $\psi: \Bbb{R}^d \to \mathcal{A}$ denote the induced linear isomorphism
 \[
  \psi ( (x_1,\ldots,x_d)^T) = \sum_{i=1}^d x_i a_i~. 
 \] Let $\Psi: \mathcal{A} \to \mathfrak{gl}(d,\Bbb{R})$ denote the associated linear map satisfying for all $\tilde{a} \in \mathcal{A}$ and for all $x \in \Bbb{R}^d$:
 \[
 \psi^{-1}(\tilde{a} \cdot \psi(x)) = \Psi(\tilde{a}) \cdot x~.  
 \]
 \begin{enumerate}[label=(\alph*)]
 \item The set  
\[
 S = \{ I_d + \Psi(a)^T : a \in \mathcal{N} \}
\] is a shearing subgroup, with associated Lie algebra given by
\[
 \mathfrak{s} = \{ \Psi(a)^T : a \in \mathcal{N} \}~. 
\]
\item Defining $X_i = \Psi(a_i)^T$, for $i=1,\ldots,d$, we get that $X_1$ is the identity matrix, and $X_2,\ldots,X_d$ is the canonical basis of $\mathfrak{s}$ in the sense of Lemma  \ref{lem:char_admiss}. 
\item Let $(d_{i,j,k})_{1 \le i,j,k \le d}$ denote the structure constants associated to the basis, defined by the equations
\begin{equation} \label{eqn:defn_rel}
\mbox{for all } 1 \le i, j \le d~:~a_i a_j = \sum_{k=1}^d d_{i,j,k} a_k~.
\end{equation}
Then 
\begin{equation} \label{eqn:Xi_sc}
 X_i = \left( \begin{array}{cccc} d_{i,1,1} & d_{i,1,2} & \ldots & d_{i,1,d} \\ \vdots & \vdots & \vdots & \vdots \\ d_{i,d,1} & d_{i,d,2} & \ldots & d_{i,d,d} \end{array} \right)~.
\end{equation}
 \item We note the following nontrivial properties of the $d_{i,j,k}$, valid for all $1\le i,j,k \le d$: 
 \[ d_{i,j,k} = d_{j,i,k}~, ~d_{1,j,k} = \delta_{j,k}~,~ d_{i,j,k} = 0 \mbox{ whenever } k \le \max(i,j) ~.\] 
 In particular, we get for $2 \le j \le d$
 \begin{equation} \label{defn:struct_Xi}
  X_i = E_{1,i} +\sum_{j=2}^{d}  \sum_{k=j+1}^d d_{i,j,k} E_{j,k}~.
 \end{equation}
  \end{enumerate}
\end{lemma}
\begin{proof}
 We start with part (c). Since multiplication with $a_1 = 1_{\mathcal{A}}$ is the identity operator, the statement about $X_1$ is clear. 
 Let $1 \le i,j \le d$. By definition of $\psi$, we have $\psi(e_j) = a_j$, and hence by definition of $\Psi$
 \begin{eqnarray*}
  \Psi(a_i) e_j & = &  \psi^{-1}(a_i \cdot \psi(e_j))  = \psi^{-1}(a_i a_j) \\
   & = & \psi^{-1} \left(\sum_{k=1}^d d_{i,j,k} a_k  \right) = \sum_{k=1}^d d_{i,j,k} e_k~. 
 \end{eqnarray*}
Hence the $j$th column of $\Psi(a_i)$ is the vector $(d_{i,j,1},\ldots,d_{i,j,d})^T$, and its transpose is the $j$th row of $\Psi(a_i)^T$. This shows 
 (\ref{eqn:Xi_sc}). 
 
 Now, with (c) established, the equation
 \[
  a_1 a_j = a_j 
 \] for $i=2,\ldots,d$, yields that $d_{1,j,k} = \delta_{j,k}$, which also takes care of part (b). Furthermore, the fact that $a_i a_j = a_j a_i$ ensures that $d_{i,j,k} = d_{j,i,k}$. 
 Finally, recall that $\mathcal{A} \mathcal{I}_{i} \subset \mathcal{I}_{i+1}$ by (\ref{eqn:ideals_proper}), which entails $a_i a_j \in \mathcal{I}_{i+1}$, and thus $d_{i,j,k} = 0$ whenever $k\le i$.
 Since $d_{i,j,k} = d_{j,i,k}$, we then obtain more generally that $k \le \max(i,j)$ entails $d_{i,j,k} = 0$.  Now equation (\ref{defn:struct_Xi}) is clear, and (d) is shown.
 
 In order to prove (a), we first note that $\Psi$ is a homomorphism of associative algebras, hence $\mathfrak{s} = \Psi(\mathcal{N})^T$ is a commutative associative matrix algebra. In particular, it is also an abelian Lie-subalgebra. Furthermore, the relation $d_{i,j,k} = 0$ whenever $k \le \max(i,j)$ ensures that the basis $X_2,\ldots,X_d$ consists of strictly upper triangular matrices.  In addition, 
 $d_{1,j,k} = \delta_{j,k}$ entails that $X_2,\ldots,X_d$ is indeed a canonical basis, and thus Lemma \ref{lem:char_admiss} gives that the associated Lie group is a shearing subgroup. Now part (a) of Lemma \ref{lem:desc_sg} yields (a) of the current lemma, and (b) is also shown.
 \end{proof}

\begin{remark}
 It is natural to ask whether the construction of shearing subgroups $S$ from irreducible commutative associative algebras $\mathcal{A}$, as described in Lemma \ref{lem:const_admis}, is exhaustive. The answer is  yes. To see this, consider the Lie algebra $\mathfrak{s}$ of a shearing subgroup $S$. Let $X_2,\ldots,X_d$ be the canonical basis of $\mathfrak{s}$. Since $X_j$ is strictly upper triangular, and the first row of $X_i$ equals $e_i^T$, it follows that the first $i$ entries of the first row of $X_i X_j$ vanish. This product is again in the span of the $X_k$, hence 
 \[
  X_i X_j = \sum_{k > i} d_{i,j,k} X_k~,
 \] with suitable coefficients $d_{i,j,k}$. But the fact that the sum on the right-hand side starts with $k=i+1$ shows that the basis $X_2,\ldots,X_d$ is a Jordan-H\"older basis of the nilpotent associative matrix algebra $\mathfrak{s}$. If one now applies the procedure from Lemma \ref{lem:const_admis} (with $a_i = X_i$), direct calculation allows to verify that $\Psi(X)^T = X $ for all $X \in \mathfrak{s}$. 
Hence every shearing subgroup arises from the construction in Lemma \ref{lem:const_admis}.

In particular, the observations concerning the structure constants $d_{i,j,k}$ made in part (d) of Lemma \ref{lem:const_admis} also apply to the $d_{i,j,k}$ in Lemma \ref{lem:char_diag}.
\end{remark}

\begin{remark}
 A further benefit of the above construction of shearing groups via associative algebras is that it settles the question of conjugacy as a byproduct. By Theorem 13 in \cite{Fu98} and the remarks prior to that result, one sees that two shearing subgroups $S_1$ and $S_2$ are conjugate iff their Lie algebras are isomorphic as associative algebras. 
 
 In particular, following the observation made in \cite[Theorem 15]{Fu98}, in dimension $d \ge 7$ there exist uncountably many nonconjugate shearing subgroups. 
\end{remark}
\subsection{An inductive approach to shearlet dilation groups}
\label{subsect:inappr_sheardil}

For possible use in inductive proof strategies, we note a further consequence of the block structure: 
\begin{proposition}
 Let $H = \pm DS< Gl(d,\Bbb{R})$ denote a shearlet dilation group, with $d \ge 3$, and let 
 \[
 H_1 = \left\{ h' \in {\rm GL}(d-1,\Bbb{R}): \exists\, h \in H, z \in \Bbb{R}^{d-1}, s\in\Bbb{R}\backslash\left\{0\right\} \mbox{ with }
  h = \left( \begin{array}{cc} h' & z \\ 0 & s \end{array} \right) ~\right\}~. 
 \]
Then $H_1$ is a shearlet dilation group as well. 

Conversely, the elements of $H$ can be described in terms of $H_1$ as follows: There exists a map $y: H_1 \to \Bbb{R}^{d-1}$ such that we can write each $h \in H$ uniquely as 
\[
 h(h_1,r) = \left( \begin{array}{ccccc} & & & & r \\ & & & & y_1(h_1) \\ & & h_1 & & y_2(h_1) \\ & & & &  \vdots \\ &  & & &  y_{d-2}(h_1) \\ 
 0 & \ldots & \ldots & 0 & y_{d-1}(h_1) \end{array} \right) ~, 
\]
with $ h_1 \in H_1, r \in \Bbb{R}$.
\end{proposition}

\subsection{Examples}

\label{subsect:ex}

As a result of the previous subsections, we obtain the following general procedure for the systematic construction of shearlet dilation groups:
\begin{enumerate}
 \item Fix a nilpotent associative algebra $\mathcal{N}$.
 \item Pick a Jordan-H\"older basis $a_2,\ldots,a_d$ of $\mathcal{N}$, and compute the canonical basis $X_2,\ldots,X_d$ of the Lie algebra $\mathfrak{s}$ of the associated shearing subgroup. Note that this amounts to determining the structure constants $(d_{i,j,k})_{1\le i,j,k \le d}$. The shearing subgroup is then determined as $S = I_d + \mathfrak{s}$.  
 \item In order to determine the diagonal scaling groups that are compatible with $S$, set up and solve the linear system (\ref{eqn:char_diag}) induced by the nonvanishing $d_{i,j,k}$. 
\end{enumerate}

We will now go through this procedure for several examples or classes of examples. 

\begin{example}
 We start out with the simplest case of a nilpotent algebra $\mathcal{N}$ of dimension $d-1$, namely that of nilpotency class 2. Here one has $ab = 0$ for any $a,b \in \mathcal{N}$, and it is clear that for two such algebras, any linear isomorphism is an algebra isomorphism as well. Picking any basis $a_2, \ldots, a_n$ of $\mathcal{N}$, we obtain $X_i = E_{1,i}$. In particular, the linear system (\ref{eqn:char_diag}) is trivial. Hence any one-parameter diagonal group can be used as scaling subgroup. We thus recover the groups described in (\ref{eqn:clsh_ddim}).
\end{example}

\begin{example} \label{ex:toeplitz}
 Another extreme class of nilpotent algebras of dimension $d$ is that of nilpotency class $d$. Here there exists $b \in \mathcal{N}$ with $b^{d-1} \not= 0$.
 This implies that $b,\ldots,b^{d-1}$ are linearly independent, and then it is easily seen that $a_i = b^{i-1}$, for $i=2,\ldots,d$, defines a Jordan-H\"older basis of $\mathcal{N}$. 
 In this example, the defining relations read
 \begin{equation} \label{eqn:toeplitz_rel}
  a_i a_j = a_{i+j-1},~2 \le i,j,i+j-1 \le d,
 \end{equation} and the resulting canonical Lie algebra basis is then determined as 
 \begin{eqnarray*}
 X_2  & = & \left( \begin{array}{cccccc} 0 & 1 &  &  &  &  \\ 
  & 0 & 1 &  & & \\  &  & \ddots & \ddots & & \\ &  &   &\ddots & \ddots & \\ & &  & & 0 & 1 \\
 & & & & & 0 \end{array} \right)
 ~,~    
  X_3 =  \left( \begin{array}{cccccc} 0 & 0 & 1 &  & & \\ 
  & 0 & 0 & 1 & &  \\  &  & \ddots & \ddots & \ddots & \\ &  &   & 0 & 0 & 1 \\ & &  & & 0 & 0 \\
 & & & & & 0 \end{array} \right)~,\ldots,~ \\
 \end{eqnarray*}
 \[
 X_d  =   \left( \begin{array}{cccccc}   0 & \ldots & \ldots  & \ldots & 0 & 1 \\
                  &   &  &   &  &  0 \\
                  &   &  & & & \vdots \\
                  &  &  & \mathbf{0}  &  &  \vdots\\
                  &  &  &  &      & \vdots  \\
                  &  &  &  &    &   0
            \end{array} \right)~.
            \]
Thus we see that the resulting shearing subgroup is that of the Toeplitz shearlet group from (\ref{eqn:tsh_ddim}). The linear system (\ref{eqn:char_diag}) becomes
\[
 \mu_{i}+\mu_{j} = \mu_{i+j-1}, \mbox{ for } 2 \le i,j, \, i+j-1 \le d.
\] It is easy to see that all solutions of this system are given by 
\[
 \mu_j = (j-1) \delta,~j=2,\ldots,d
\] with $\delta$ an arbitrary real parameter. Thus the scaling subgroups compatible with the Toeplitz dilation group are precisely given by 
\[ \exp(\Bbb{R} {\rm diag}(1,1+\delta,\ldots,1+(d-1)\delta))\,, \] with $\delta \in \mathbb{R}$ arbitrary. 
\end{example}

\begin{remark}
For $d=3$, the two above listed cases are all possible examples of shearing subgroups, and not even  just up to conjugacy. In particular, we find that all shearing subgroups in dimension 3 are compatible with anisotropic dilations. 
\end{remark}

We now turn to the shearing subgroups in dimension 4, with focus on the groups not covered by (\ref{eqn:clsh_ddim}) and (\ref{eqn:tsh_ddim}).  
\begin{example}
Since the nilpotency classes $n=2,4$ are already covered by the previous examples, the remaining 4-dimensional cases of irreducible algebras $\mathcal{A}$ all have nilpotency class 3. It is shown in \cite{Fu98} that $\mathcal{A} \cong \Bbb{R}[Y_1,Y_2]/(Y_1^3, Y_2^2- \alpha Y_1^2,Y_1 Y_2)$, with $\alpha \in \{ -1,0,1 \}$. Here, $\Bbb{R}[Y_1,Y_2]$ denotes the algebra of polynomials with real coefficients and indeterminates $Y_1,Y_2$, and $\mathcal{J} = (Y_1^3,Y_2^2-\alpha Y_1^2, Y_1 Y_2)$ denotes the ideal generated by the three polynomials. 
Then the nilradical $\mathcal{N}$ is generated by $Y_1+\mathcal{J},Y_2+\mathcal{J}$. We choose the basis $a_2 = Y_1 + \mathcal{J},~ a_3 = Y_2 + \mathcal{J},~ a_4 = Y_1^2 + \mathcal{J}$, and obtain as the only nonzero relations
\[
 a_2^2 = a_4~,~ a_3^2 = \alpha a_4 ~.
\] This allows to conclude that $a_2,a_3,a_4$ is indeed a Jordan-H\"older basis. Following Lemma \ref{lem:const_admis} (c), we can read off the canonical basis of the associated shearing subgroup as 
\[
X_2 =  \left( \begin{array}{cccc} 0  & 1  & 0  & 0 \\  0  & 0 &  0 & 1 \\  0 & 0  &0  & 0  \\ 0 &0  & 0 & 0 \end{array} \right) ~,~ X_3 = \left( \begin{array}{cccc} 0  & 0 & 1 & 0\\ 0&0 & 0& 0\\ 0& 0& 0& \alpha \\ 0&0 &0& 0\end{array} \right)~,~
X_4 =  \left( \begin{array}{cccc} 0& 0  & 0& 1 \\ 0 & 0& 0& 0 \\ 0& 0& 0& 0\\ 0& 0& 0& 0\end{array} \right). 
\] We next determine the compatible scaling subgroups. 
In the case $\alpha \not=0$, we obtain the system of equations 
 \[
  2 \mu_2 = \mu_4,~ 2 \mu_3 = \mu_4. 
 \] Thus the infinitesimal generators of scaling subgroups are of the form $Y = {\rm diag}(1,1+\delta,1+\delta,1+2\delta)$, with $\delta \in \Bbb{R}$ arbitrary. 
 
 In the case $\alpha=0$, we only get one equation, namely 
 \[
  2 \mu_2 = \mu_4,
 \] showing that here the compatible infinitesimal generators are of the form $Y = {\rm diag}(1,1+\delta_1,1+\delta_2,1+2\delta_1)$, with $\delta_1,\delta_2 \in \Bbb{R}$ arbitrary. 
\end{example}

Finally, we give an example of a shearing subgroup which is only compatible with isotropic scaling. It is based on the same algebra as Example \ref{ex:toeplitz} (with $d=4$), and as a result the associated shearing subgroups are conjugate. Recall that the groups in Example \ref{ex:toeplitz} are compatible with anisotropic scaling. This illustrates an important, somewhat subtle point: While the precise choice of Jordan-H\"older basis in the procedure described in Lemma \ref{lem:const_admis} is immaterial if one is just interested in guaranteeing the shearing subgroup property, it may have a crucial influence on the availability of compatible {\em anisotropic} scaling subgroups. 
\begin{example} \label{ex:not_anisotropic}
Let $\mathcal{A} = \Bbb{R}[X]/(X^4)$. We use the Jordan-H\"older algebra $a_2 = X+ X^2 + (X^4)$, $a_3 = X^2  +(X^4)$, $a_4 = X^3 + (X^4)$. This leads to the following nonzero relations
\[
 a_2^2 = a_3 + 2 a_4,~ a_2 a_3 = a_4,
\]  which gives rise to the basis 
\[
X_2 =  \left( \begin{array}{cccc} 0  & 1  & 0  & 0 \\  0  & 0 &  1 & 2 \\  0 & 0  &0  & 1  \\ 0 &0  & 0 & 0 \end{array} \right) ~,~ X_3 = \left( \begin{array}{cccc} 0  & 0 & 1 & 0\\ 0&0 & 0& 1\\ 0& 0& 0& 0 \\ 0&0 &0& 0\end{array} \right)~,~
X_4 =  \left( \begin{array}{cccc} 0& 0  & 0& 1 \\ 0 & 0& 0& 0 \\ 0& 0& 0& 0\\ 0& 0& 0& 0\end{array} \right). 
\]
Now the nonzero entries in the matrix $X_2$ imply that the linear system (\ref{eqn:char_diag}) contains the equations 
\[
 2 \mu_2 = \mu_3,~2 \mu_2 = \mu_4, ~\mu_2+\mu_3 = \mu_4.
\] The first two equations imply $\mu_3 = 2 \mu_2 = \mu_4$, and then the third equation yields $\mu_2 = 0$. Hence this shearing subgroup is only compatible with isotropic scaling. 
\end{example}

\section{Anisotropic scaling and wavefront set characterizations}
\label{sec:anis_wave}
In this section we investigate the suitability of the various groups for microlocal analysis. The idea is to verify the criteria derived in \cite{FeFuVo} that allow to establish the suitability of a dilation group for the characterization of the wavefront set via wavelet coefficient decay. As it will be seen, this property only depends on the scaling subgroup. 

\subsection{Criteria for wavefront set characterization}
\label{subseq:crit_wave}

Throughout this subsection $H$ is an irreducibly admissible matrix group, i.e. its dual action has a single open orbit
$\mathcal{O} \subset \Bbb{R}^d$, with associated compact fixed groups.  We use $V \Subset \mathcal{O}$ to denote that
the closure of $V$ inside $\mathcal{O}$ is compact. 

Given $R>0$ and $x\in\Bbb{R}^{d}$, we let $B_{R}(x)$ and $\overline{B_{R}}\left(x\right)$
denote the open/closed ball with radius $R$ and center $x$, respectively. We let $S^{d-1}\subset\Bbb{R}^{d}$ denote the unit sphere. By a neighborhood of $\xi\in S^{d-1}$, we will always mean a {\em relatively open} set $W\subset S^{d-1}$ with $\xi\in W$. Given $R>0$ and an open set $W\subset S^{d-1}$, we let 
\begin{eqnarray*}
&&C(W):=\left\{ r\xi' : \xi'\in W, r>0\right\} =\left\{\xi\in\Bbb{R}^{d}\setminus\left\{ 0\right\}:\frac{\xi}{\left|\xi\right|}\in W\right\},\\
&&C(W,R):=C(W)\setminus\overline{B_{R}}(0).
\end{eqnarray*}
Both sets are clearly open subsets of $\Bbb{R}^{d}\setminus\left\{ 0\right\} $
and thus of $\Bbb{R}^{d}$.

Given a tempered distribution $u$, we call $(x,\xi)\in\Bbb{R}^{d}\times S^{d-1}$
a {\em regular directed point of $u$} if there exists $\varphi\in C_{c}^{\infty}(\Bbb{R}^{d})$,
identically one in a neighborhood of $x$, as well as a $\xi$-neighborhood
$W\subset S^{d-1}$ such that for all $N\in\Bbb{N}$ there exists
a constant $C_{N}>0$ with 
\begin{equation}
\mbox{for all }\xi'\in C\left(W\right)~:~\left|\widehat{\varphi u}(\xi')\right|\le C_{N}(1+|\xi'|)^{-N}.\label{eqn:decay_cond_reg_dir}
\end{equation}

We next formally define the sets $K_{i}$ and $K_{o}$ which will
allow to associate group elements to directions. 
\begin{definition}
Let $\emptyset\neq W\subset S^{d-1}$ be open with $W\subset\mathcal{O}$
(which implies $C\left(W\right)\subset\mathcal{O}$). Furthermore,
let $\emptyset\neq V\Subset\mathcal{O}$ and $R>0$. We define 
\[
K_{i}(W,V,R):=\left\{ h\in H : h^{-T}V\subset C(W,R)\right\} 
\]
as well as 
\[
K_{o}(W,V,R):=\left\{ h\in H :  h^{-T}V\cap C(W,R)\not=\emptyset\right\} .
\]
If the parameters are provided by the context, we will simply write
$K_{i}$ and $K_{o}$. Here, the subscripts $i/o$ stand for ``inner/outer''.
\end{definition}

We now define what we mean by dilation groups characterizing the wavefront set. We first extend the continuous wavelet transform to the space of tempered distributions. I.e., we use $\mathcal{W}_\psi u$, for a Schwartz wavelet $\psi$ and a tempered distribution $u$. 
\begin{definition}
 The dilation group $H$ {\em characterizes the wavefront set} if there exists a nonempty open subset $V \Subset  \mathcal{O}$ with the following property:  For all $0 \not= \psi \in \mathcal{S}(\Bbb{R}^d)$ with ${\rm supp}(\widehat{\psi}) \subset V$, for every $u \in \mathcal{S}'(\Bbb{R}^d)$ and all $(x,\xi) \in \Bbb{R}^d \times (\mathcal{O} \cap S^{d-1})$, the following statements are equivalent:
\begin{enumerate}[label=(\alph*)]
 \item $(x,\xi)$ is a regular directed point of $u$.
 \item { There exists a neighborhood
$U$ of $x$, some $R>0$ and a $\xi$-neighborhood $W\subset S^{d-1}$
such that for all $N\in\Bbb{N}$ there exists a constant $C_{N}>0$ such that for all $y\in U$, and for all $h\in K_{o}(W,V,R)$ the following estimate holds: 
\[
|W_{\psi}u(y,h)|\le C_{N}\|h\|^{N}.
\]}
 \end{enumerate}
\end{definition}

Note that the definition excludes a set of directions $\xi$ from the analysis of the wavefront set, namely the directions not contained in $\mathcal{O} \cap S^{d-1}$. These directions always constitute a set of measure zero. Recall from Proposition \ref{prop:open_orbit} that  in the case of shearlet dilation groups, this exceptional set is given by $(\{ 0 \} \times \Bbb{R}^{d-1}) \cap S^{d-1}$. 

We next recall the sufficient conditions for dilation groups that characterize the wavefront set, as established in \cite{FeFuVo}. The first one is related to the problem that one would like to interpret the norm as a scale parameter. 
\begin{definition}
\label{defn:micro_regular}Let $\xi\in\mathcal{O}\cap S^{d-1}$ and
$\emptyset\neq V\Subset\mathcal{O}$. The dual action is called {\em $V$-microlocally
admissible in direction $\xi$} if there exists a $\xi$-neighborhood
$W_{0}\subset S^{d-1}\cap\mathcal{O}$ and some $R_{0}>0$ such that
the following hold: 
\begin{enumerate}
\item \label{enu:NormOfInverseEstimateOnKo}There exist $\alpha_{1}>0$
and $C>0$ such that 
\[
\|h^{-1}\|\le C\cdot\|h\|^{-\alpha_{1}}
\]
holds for all $h\in K_{o}(W_{0},V,R_{0})$. 
\item \label{enu:NormIntegrability}There exists $\alpha_{2}>0$ such that
\[
\int_{K_{o}(W_{0},V,R_{0})}\|h\|^{\alpha_{2}}\,{\rm d}h<\infty.
\]

\end{enumerate}
The dual action is called {\em microlocally admissible in direction
$\xi$} if it is $V$-microlocally admissible in direction $\xi$
for some $\emptyset\neq V\Subset\mathcal{O}$. 
\end{definition}

The second important condition is contained in the following definition. It can be understood as formalizing the ability of the associated wavelet systems to be able to make increasingly fine distinctions between different directions, as the scales go to zero. 
\begin{definition}
\label{def:ConeApproximationProperties}Let $\xi\in\mathcal{O}\cap S^{d-1}$ and  $\emptyset\neq V \Subset\mathcal{O}$. The dual action has
the {\em $V$-cone approximation property at $\xi$} if for
all $\xi$-neighborhoods $W\subset S^{d-1}$ and all $R>0$ there
are $R'>0$ and a $\xi$-neighborhood $W'\subset S^{d-1}$ such that
\[
K_{o}(W',V,R')\subset K_{i}(W,V,R).
\]
\end{definition}

We now have the following \cite[Corollary 4.9]{FeFuVo}:
\begin{theorem} \label{thm:char_wfset}
 Assume that the dual action is $V$-microlocally admissible at some $\xi_0 \in \mathcal{O}$ and has the $V$-cone approximation property at $\xi_0$, for some nonempty open subset $V \subset \mathcal{O}$. Then $H$ characterizes the wavefront set. 
\end{theorem}

\begin{remark}
The property of characterizing the wavefront set is linked to anisotropic scaling, in the following sense: If $H$ characterizes the wavefront set, then 
\[ H \cap \Bbb{R}^+ \cdot I_d  = \{   I_d  \} \,, \] by \cite[Lemma 4.10]{FeFuVo}. Hence if $H$ is a shearlet dilation group characterizing the wavefront set, its shearing subgroup must admit at least one anisotropic compatible scaling subgroup. This excludes the shearing group constructed in Example \ref{ex:not_anisotropic}. 

Theorem \ref{thm:char_wfset} therefore implies that every group failing the anisotropy criterion $H \cap \Bbb{R}^+ \cdot I_d = \{   I_d \}$ must necessarily fail either the microlocal admissibility or the cone approximation property. It is in fact the latter that breaks down, as noted in \cite[Lemma 4.4]{FeFuVo}.  

These considerations highlight the importance of understanding when a given shearing groups admits anisotropic scaling.
\end{remark}

\subsection{Characterization of the wavefront set for shearlet dilation groups}
\label{subsec:charact_wavefront}
We can now state a very general theorem concerning the ability of shearlet groups to characterize the wavefront set. Note that there are no conditions on the shearing subgroups. 
\begin{theorem} \label{thm:wfset_general_shearlet}
Let $H$ be a shearlet dilation group and let $Y = {\rm diag}(1,\lambda_2,\ldots,\lambda_d)$ denote the infinitesimal generator of the scaling subgroup. If
 $0 < \lambda_i < 1$ holds, for all $2 \le i \le d$, then $H$ characterizes the wavefront set. 
\end{theorem}

\begin{remark}
 We can now quickly go through the examples of shearing subgroups in Subsection \ref{subsect:ex} and show that for most cases, there exists a compatible anisotropic scaling subgroup that allows to characterize the wavefront set. Writing $\lambda_i = 1 + \mu_i$ as in Lemma \ref{lem:char_admiss}, the condition  from Theorem \ref{thm:wfset_general_shearlet} translates to $-1 < \mu_i < 0$, for $2 \le i \le d$. 
 Apart from the group in Example \ref{ex:not_anisotropic}, which was specifically constructed to not allow any anisotropic scaling, all other shearing groups can be combined with a compatible scaling group in such a way that the resulting shearlet transform fulfills the conditions of Theorem \ref{thm:wfset_general_shearlet}, and therefore characterizes the wavefront set. Note that this was previously known only for the original shearlet group \cite{KuLa,FeFuVo}.
In particular, we may combine the Toeplitz shearing subgroup with the scaling subgroup with 
exponents $(1,1-\delta,\ldots,1-(d-1)\delta)$, and choosing $\delta \in (0, 1/(d-1))$ guarantees that the Toeplitz shearlet transform characterizes the wavefront set.  
\end{remark}

The proof of the Theorem amounts to verifying the cone approximation property and microlocal admissibility of the dual action, and this will be carried out in the following two propositions.
For the remainder of this section, we fix a shearlet dilation group $H$ with infinitesimal generator ${\rm diag}(1,\lambda_2,\ldots,\lambda_d)$ of the scaling subgroup. 
We let $\lambda_{\max} = \max_{i\ge 2} \lambda_i$, and $\lambda_{\min} = \min_{i\ge 2} \lambda_i$.

\begin{proposition}
If $\lambda_{\max}<1$, there
exists an open subset $\emptyset\neq V \Subset\mathcal{O}$ such that
the dual action of $H$ on the orbit $\mathcal{O}$ has the
$V$-cone approximation  property at $(1,0,\ldots,0)^T\in S^{d-1}\cap\mathcal{O}$.
\end{proposition}
\begin{proof}
We will employ the structural properties of shearing subgroups derived in Section \ref{sect:gen_shearlets}. 
We let $S$ and $D$ denote the shearing and scaling subgroups of $H$, respectively. The infinitesimal generator
of $D$ is a diagonal matrix with the entries $1, \lambda_2,\ldots,\lambda_d$. We let $X_2,\ldots,X_d$ denote the canonical
basis of $\mathfrak{s}$, consisting of strictly upper triangular matrices $X_i$. By Lemma \ref{lem:desc_sg}, each $h \in S$ is
uniquely described by
\[
 h = h(t,1) = I_d + \sum_{i=2}^d t_i X_i\,,
\] 
where $t= (t_2,\ldots,t_d)^T$ denotes the vector of first row entries of $h(t,1)$. 
For $H$, we thus obtain the global chart
\[
h(t,a)=\left(I_d + \sum_{i=2}^d t_i X_i
\right)
{\rm sgn}(a){\rm diag}(|a|,|a|^{\lambda_2},\ldots,|a|^{\lambda_d})
\in  GL(d,\Bbb{R}),
\] 
with $(t,a) \in \Bbb{R}^{d-1} \times \Bbb{R}^\times$.

For the purpose of the following computations, it is possible and beneficial to slightly modify this construction and replace $h(t,1)$ by its inverse. 
Thus, every $h\in H$ can be written (uniquely) as $h=\pm h(t,1)^{-1} h(0,a)$ with $t\in\Bbb{R}^{d-1}$ and $a\in (0,+\infty)$. The dual action is then given by
\begin{equation} \label{eq:1}
(h^{-1})^T =\pm (h(0,a)^{-1}h(t,1) )^T = \pm
\left(I_d + \sum_{i=2}^d t_i X_i^T\right) h(0,a^{-1}),
\end{equation}
where by construction 
\begin{equation}\label{eq:2}
I_d+\sum_{i=2}^d t_i X_i^T = 
  \begin{pmatrix}
  1 &      0^T\\
  t  &       I_{d-1} + A(t)^T              
  \end{pmatrix},
\end{equation}
with $A(t)$ being a $(d-1)\times (d-1)$ strictly lower-triangular matrix satisfying
\begin{equation}
  \label{eq:5}
  \lVert A(t)\rVert \leq C |t|
\end{equation}
with a constant $C$ depending only on $H$.

We now parametrise the open orbit $\mathcal O$ by  the global chart
provided by affine coordinates
\[
\Omega: \Bbb{R}^\times\times\Bbb{R}^{d-1} \to \mathcal O\qquad
\Omega(\tau,v)=\tau (1,v^T)^T,
\]
and $S^{d-1}\cap \mathcal O$ by the corresponding diffeomorphism to its image
\[ \omega:\Bbb{R}^{d-1}\to S^{d-1}\cap \mathcal{O}, \qquad \omega(v)=
\frac{(1,v^T)^T}{\sqrt{1+|v|^2}}.
\]

Given $\epsilon>0$, we set
\[
W_{\epsilon}=\{v\in\Bbb{R}^{d-1} : |v| < \epsilon \}=B_\epsilon(0),
\]
since $\{W_{\epsilon}:\epsilon>0\}$ is a neighbourhood basis of the origin in
$\Bbb{R}^{d-1}$ and $\{\omega(W_{\epsilon}) :\epsilon>0\}$ is a neighbourhood basis of
$\xi_0=(1,0,\ldots,0)\in S^{d-1}\cap \mathcal O$.  

Furthermore, for fixed $0<\tau_1<\tau_2$ and $\epsilon_0>0$  the set
\[
V= \Omega(\, (\tau_1,\tau_2)\times W_{\epsilon_0}\, )
\]
is an open subset with $V\Subset \mathcal O$.

Given $h\in H$, as in~(\ref{eq:1}), and $\xi\in V$, then
$\xi=\Omega(\tau,v)$ with $\tau_1<\tau<\tau_2$ and $v\in W_{\epsilon_0}$, and we get
\begin{align*}
(h^{-1})^T \xi & =  \pm  \tau  \begin{pmatrix}
  1 &      0^T\\
  t  &      I_{d-1}+ A(t)^T              
  \end{pmatrix}  \begin{pmatrix}
  a^{-1}   \\
  v'           
  \end{pmatrix}  = \pm  a^{-1} \tau 
 \begin{pmatrix}
 1 \\
   t+(I_{d-1} + A(t)^T)v''         
  \end{pmatrix}
\end{align*}
where $v',v''\in\Bbb{R}^{d-1}$ have components given by $v'_i =
a^{-\lambda_i} v_i$ and $v''_i= a^{1-\lambda_i} v_i $ for all $i=2,\ldots, d$. Hence
\[
(h^{-1})^T(V)=\Omega\left(\, (\pm a^{-1} \tau_1,\pm a^{-1} \tau_2)\times
  (t+(I_{d-1} +A(t)^T)W_{\epsilon_0}^a)\,\right)
   \]
where $ W_{\epsilon_0}^a=\{v''\in\Bbb{R}^{d-1} : v''_i=a^{1-\lambda_i}v_i,\
|v|<\epsilon_0\}$. 

Fix now $R>0$ and a neighborhood  $W\subset S^{d-1}\cap\mathcal O$ of
$\xi_0$. Without loss of  generality we can assume that
$W=\omega(W_\epsilon)$ for some $\epsilon>0$. Furthermore, since
\[
 (R,+\infty)\times W_\epsilon \subset \Omega^{-1}( C(\omega(W_\epsilon),R)) \subset
 (\frac{R}{\sqrt{1+\epsilon^2}},+\infty)\times W_\epsilon\subset
 (\frac{R}{2},+\infty)\times W_\epsilon ,
\]
where the last inclusion on the right holds if $\epsilon\leq 1$, then the
$V$-cone approximation property holds true if there exist $R'>0$ and 
$0<\epsilon' \leq 1$ such that for all $h\in H$ satisfying
\begin{subequations}
  \begin{equation}
    (h^{-1})^T(V)\cap \Omega\left(\, (\frac{R'}{2},+\infty)\times
      W_{\epsilon'}\,\right) \not=\emptyset,\label{eq:9a}
  \end{equation}
it holds that
  \begin{equation}
    (h^{-1})^T(V)\subset  \Omega\left(\,(R,+\infty)\times
      W_\epsilon\,\right).\label{eq:9b}
  \end{equation}
\end{subequations}
Take $R'>0$ and  $0<\epsilon' <\sqrt{3}$, which we will fix later on as functions of
$R$ and $\epsilon$,  and $h\in H$ as in (\ref{eq:1}). If $h=- (h(0,a)^{-1}h(t,1) )^T$ then
\[
 \left(\,( \frac{R'}{2},+\infty)\times W_{\epsilon'} \,\right)\cap
 \left(\, (-a^{-1} \tau_2,-a^{-1} \tau_1)\times 
  (t+(I_{d-1} +A(t)^T)W_{\epsilon_0}^a)
  \,\right)=\emptyset,
\]
so that (\ref{eq:9a}) implies that  $h=+  (h(0,a)^{-1}h(t,1))^T$ and
\[
 \left(\,( \frac{R'}{2},+\infty)\times W_{\epsilon'} \,\right)\cap \left(\, ( a^{-1} \tau_1, a^{-1} \tau_2)\times
  (t+(I_{d-1} +A(t)^T)W_{\epsilon_0}^a)
  \,\right)\not=\emptyset.
\]
Hence
\[
   R'< 2 a^{-1}\tau_2,\qquad 
   W_{\epsilon'} \cap \left(t+(\operatorname{Id}_{d-1} +A(t)^T)W_{\epsilon_0}^a \right)\not=\emptyset.
\]
If we choose $R'>2\tau_2$, the  first inequality gives 
\begin{equation}
  \label{eq:6}
  a< \frac{2\tau_2}{R'}<1,
\end{equation}
and, since $a<1$, setting
$\lambda_{\max}=\max\{\lambda_2,\ldots,\lambda_d\}$, clearly
\begin{equation}
  \label{eq:3}
 W_{\epsilon_0}^a\subset  W_{\epsilon_0 a^{1-\lambda_{\max}}}.
\end{equation}
By the above inclusion, since $W_{\epsilon'} \cap
\left(t+(I_{d-1} +A(t)^T)W_{\epsilon_0}^a
\right)\not=\emptyset$, then  there exists $\xi\in W_{\epsilon_0
  a^{1-\lambda_{\max}}} $ such that  
$|t+\xi+A(t)^T\xi|<\epsilon'$.  Hence, triangle inequality,~(\ref{eq:5})
and ~(\ref{eq:6}) give
\begin{align*}
|t| & <  \epsilon'+(1+\lVert A(t)^T\rVert )|\xi| \leq \epsilon'+ (1+C |t|) 
a^{1-\lambda_{\max}}\epsilon_0  \\ & \leq \epsilon'+  (\frac{2\tau_2}{R'})^{1-\lambda_{\max}} (1+C |t|) \epsilon_0  \leq  2\epsilon'+ \frac{1}{2} |t|,
\end{align*}
where the last inequality holds true  provided that
\begin{equation}
  \label{eq:4}
  R'>2\tau_2\max\{ 1,(\frac{\epsilon_0}{\epsilon'})^{\frac{1}{1-\lambda_{\max}}},
  (2C\epsilon_0)^{\frac{1}{1-\lambda_{\max}}}\}. 
\end{equation}
Hence, if~\eqref{eq:9a} holds true with~$R'$ satisfying~\eqref{eq:4}, then 
\begin{subequations}
  \begin{align}
    a & <  \frac{2\tau_2}{R'}<1 \label{eq:6a}\\
   |t| & <  4\epsilon' 
  \label{eq:6b} \\
(\frac{2\tau_2}{R'})^{1-\lambda_{\max}} \epsilon_0 & < \min\{\epsilon' ,\frac{1}{2C}\} \label{eq:6c}.
  \end{align}
\end{subequations}
The condition~\eqref{eq:9b}  is equivalent to
\[
(a^{-1} \tau_1, a^{-1} \tau_2)\times
  (t+I_{d-1} +A(t)^T)W_{\epsilon_0}^a) \subset
(R,+\infty)\times W_{\epsilon} ,
\]
which is ensured by $a^{-1} \tau_1>R$ and, recalling~\eqref{eq:3}, by
$t+(I_{d-1} +A(t)^T)W_{\epsilon_0
  a^{1-\lambda_{\max}}}\subset W_\epsilon$.

By~\eqref{eq:6a} the first condition is satisfied if $\tau_1/R>
\frac{2\tau_2}{R'}$. Taking into account~\eqref{eq:4}, it is
sufficient to assume that 
\begin{equation}
R'>
2\tau_2\max\{1,(\frac{\epsilon_0}{\epsilon'})^{\frac{1}{1-\lambda_{\max}}},(2C\epsilon_0)^{\frac{1}{1-\lambda_{\max}}},\frac{R}{\tau_1}\}.  
\label{eq:8}
\end{equation}
To ensure that $t+(I_{d-1}+A(t)^T)W_{\epsilon_0a^{1-\lambda_{\max}}}\subset W_\epsilon$, note that, for all $\xi\in W_{\epsilon_0 a^{1-\lambda_{\max}}}$, conditions~(\ref{eq:5}),~(\ref{eq:6a}), and \eqref{eq:6b} give 
\begin{align*}
|t+(I_{d-1}+A(t)^T)\xi |& \leq |t| + (1+C|t|)|\xi|\leq|t|+(1+C|t|)a^{1-\lambda_{\max}} \epsilon_0 \\
& <4\epsilon'+(1+C 4\epsilon')(\frac{2\tau_2}{R'})^{1-\lambda_{\max}}\epsilon_0 \\
& \leq 4\epsilon'+ \epsilon'+2\epsilon' =7\epsilon',
\end{align*}
where the last inequality follows from~\eqref{eq:6c}. Hence, with the choice
$\epsilon'=\min\{1,\epsilon/7\}$ and $R'$ satisfying~(\ref{eq:8})
for all $\xi\in W_{\epsilon_0 a^{1-\lambda_{\max}}}$, 
\[
|t+(I_{d-1}
+A(t)^T)\xi |<\epsilon,
\]
so that~\eqref{eq:9b} holds true for all $h\in H$ satisfying~\eqref{eq:9a}.
\end{proof}

\begin{remark}
The proof does not make use of the fact that the shearlet group
$\mathcal S$ is abelian. The proof is based only on the following two
properties of $S$
\begin{enumerate}[label=(\alph*)]
\item a global smooth chart $t\mapsto s(t)$ from $\Bbb{R}^{d-1}$ onto
  $S$; 
\item for all $t\in\Bbb{R}^d$ the dual action of $s(t)$ is of the form
\[
(s(t)^{-1})^T =
  \begin{pmatrix}
  1 &      0^T\\
  t  &       B(t)              
  \end{pmatrix}
\]
where $\lVert B(t)\rVert\leq C_1+C_2 |t|$ for a suitable choice of
$C_1$ and $C_2$.
\end{enumerate}
\end{remark}

With the cone approximation property already established, the remaining condition is quite easy to check.

\begin{proposition}
If $0 < \lambda_{\min} \le \lambda_{\max}<1$, there
exists an open subset $\emptyset\neq V \Subset\mathcal{O}$ such that
the dual action of $H$ on the orbit $\mathcal{O}$ is $V$-microlocally admissible
in direction $(1,0,\ldots,0)\in S^{d-1}\cap\mathcal{O}$.
\end{proposition}

\begin{proof}
We retain the notations from the previous proof, as well as the open set
\[
 V= \Omega(\, (\tau_1,\tau_2)\times W_{\epsilon_0}\, )~,
\] with $\tau_1 < 1 < \tau_2$. 
Since we assume $\lambda_{\max}< 1$, the cone approximation property holds, and then condition (2) of Definition~\ref{defn:micro_regular} follows from condition (1) by \cite[Lemma 4.7]{FeFuVo}. 
In addition, the cone approximation property allows to replace $K_o$ in that condition by the smaller set $K_i$. In short, it remains to prove the existence of $\alpha>0$
and $C''>0$ such that 
\[
\|h^{-1}\|\le C'' \|h\|^{-\alpha}
\]
holds for all $h\in K_{i}(\omega(W_{\epsilon}),V,R)$, for suitable $\epsilon, R>0$. In the following computations, we let $\epsilon=1$ and $R > 2$.
Now assume that $h = \pm h(t,1)^{-1} h(0,a) \in K_i(\omega(W_{\epsilon}),V,R)$, which means that $h^{-T} V \subset C(\omega(W_{\epsilon}),R)$. This implies in particular that 
\begin{eqnarray*}
 h^{-T} \left( \begin{array}{c} 1 \\ 0 \\ \vdots \\ 0 \end{array} \right)&=&\pm h(t,1)^{T} h(0,a)^{-1} \left( \begin{array}{c} 1 \\ 0 \\ \vdots \\ 0 \end{array} \right) 
 = \pm h(t,1)^T \left( \begin{array}{c} a^{-1} \\ 0 \\ \vdots \\ 0 \end{array} \right)\\
&=& \pm a^{-1} \left( \begin{array}{c} 1 \\ t \end{array} \right) \in C(\omega(W_{\epsilon}),R)\,.
\end{eqnarray*} This implies that the sign is in fact positive. Furthermore, we have $|t|\le\epsilon=1$, and then
\[
2 a^{-1} \ge \left|h^{-T}  \left( \begin{array}{c} 1 \\ 0 \\ \vdots \\ 0 \end{array} \right)\right|>R>2\,,
\]
which implies $a<1$. {By using the fact $\|h(t,1) \|\leq C(1+|t|) \le 2C \le \sqrt{2}(1+(1+C)|t|)\le \sqrt{2}(2+C)$, where $C$ was the constant from (\ref{eq:5}), we can now estimate 
\[
 \| h^{-1} \| = \| h(0,a)^{-1} h(t,1) \| \le \|h(0,a)^{-1} \| \|h(t,1) \| \le \sqrt{2}(2+C) a^{-1},
\] where we used $a<1$ and $\lambda_{\max} \le 1$ to estimate the norm of $h(0,a)^{-1}$.} In addition,
\[
 \| h \| = \| h(t,1)^{-1} h(0,a) \| \le \| h(t,1)^{-1} \| \| h(0,a) \|  \le C' a^{\lambda_{\min}}~.
\] Here we used that the set $\{ h(t,1): |t| \le 1 \} \subset H$ is compact to uniformly estimate the norm of the inverses by a suitable $C'$, and $a < 1$ to estimate the norm of $h(0,a)$. 
But these estimates combined yield
\[
 \|h^{-1} \| \le \sqrt{2}(2+C) a^{-1} \le\sqrt{2}(2+C) (C')^{1/\lambda_{\min}}  \| h \|^{-1/\lambda_{\min}}~.
\] Since we assume that $\lambda_{\min}>0$, the proof is finished.
\end{proof}

\section{Embeddings into the symplectic group}

From the analytical point of view, we saw that shearlet dilation groups are a useful tool for the characterization of the wavefront set of distributions. On the other hand, from the algebraic and geometrical points of view, these groups and the associated generalized wavelet representation exhibit an interesting link with the symplectic group and the metaplectic representation. More precisely, in this section we show that the positive part $DS$ of any shearlet dilation group $DS\cup (-DS)$ may be imbedded into the symplectic group. Note that the full group $DS\cup (-DS)$ cannot be expected to be imbedded into $Sp(d,\Bbb{R})$ \cite[Theorem 3.5]{2014-differentfaces}.  Moreover, we prove that the wavelet representation is unitarily equivalent to the metaplectic representation, provided that they are restricted to a suitable subspace of $L^2(\Bbb{R}^d)$. In fact, a much more general class of groups is allowed, see Theorem~\ref{thm:embedding}.

The relevance of the symplectic group and of the metaplectic representation in this context has already been shown in several works \cite{2013-demari-devito,2013-alberti-balletti-demari-devito,2014-alberti-demari-devito-manto,2014-differentfaces}. In particular, the argument given here generalizes \cite{2014-differentfaces}.

Let $T(d,\Bbb{R})_{+}$ denote the subgroup of ${\rm GL}(d,\Bbb{R})$ consisting of the upper triangular matrices with positive entry in position $(1,1)$, namely
\[
T(d,\Bbb{R})_{+}=\{h\in {\rm GL}(d,\Bbb{R}): \text{$h_{1,1}>0$ and $h_{i,j}=0$ for every $i>j$}\}.
\]
We consider the following subspace of $L^{2}(\Bbb{R}^{d})$:
\[
\mathcal{H}=\{f\in L^{2}(\Bbb{R}^{d}): {\rm supp}\hat{f}\subseteq\Theta_{L}\},\,\text{ where }\;\Theta_{L}=\{\xi\in\Bbb{R}^{d}:\xi_{1}\le0\}.
\]
The main result of this section reads as follows.
\begin{theorem}\label{thm:embedding}
Take $H<T(d,\Bbb{R})_{+}$. The group $G=\Bbb{R}^d\rtimes H$ may be embedded into the symplectic group, namely there exists a group embedding $\phi\colon G\to Sp(d,\Bbb{R})$. Moreover, the restriction to $\mathcal{H}$ of the quasi-regular representation $\pi$ defined in (\ref{eqn:def_quasireg}) is unitarily equivalent to $\mu \circ \phi$ restricted to $\mathcal{H}$, where $\mu$ is the metaplectic representation of $Sp(d,\Bbb{R})$.
\end{theorem}
The rest of this section is devoted to the proof of this theorem. The embedding $\phi$, the subgroup $\phi(G)$, as well as the intertwining operator between the quasi-regular representation and the metaplectic representation will be explicitly constructed.

First, we construct the subgroup $\phi(G)<Sp(d,\Bbb{R})$ and the map $\phi$. The vectorial part of $G=\Bbb{R}^d\rtimes H$ will correspond to the subspace of the $d$-dimensional symmetric matrices given by
\[
\Sigma:=\{\sigma_{b}:=\begin{pmatrix}b_{1} & b_{2}/2 & \cdots & b_{d}/2\\
b_{2}/2\\
\vdots  & & \mathbf{0} \\
b_{d}/2
\end{pmatrix}:\; b\in\Bbb{R}^{d}\}.
\]
We shall need the following preliminary result concerning the  map
\begin{equation}
\rho\colon T(d,\Bbb{R})_{+}\to GL(d,\Bbb{R}),\qquad h\mapsto\sqrt{h_{1,1}}\,h^{-T}.\label{eq:def rho}
\end{equation}

\begin{lemma}
\label{lem:rho}The map $\rho$ is a group homomorphism and for all $b\in\Bbb{R}^{d}$ and $h\in T(d,\Bbb{R})_{+}$ there holds
\begin{equation}
\rho(h)^{-T}\sigma_{b}\rho(h)^{-1}=\sigma_{hb}.\label{eq:rho}
\end{equation}
\end{lemma}
\begin{proof}
The first part is trivial, since the matrices in $H$ are upper triangular with $h_{1,1}>0$. The second part can be proven as follows.
Fix $b\in\Bbb{R}^{d}$ and $h\in T(d,\Bbb{R})_{+}$. The assertion is equivalent to
\[
h\sigma_{b}\,h^T=h_{1,1}\sigma_{hb}.
\]
Write for $i=2,\dots,d$
\[
h=\begin{pmatrix}h_{1}\\
h_{2}\\
\vdots\\
h_{d}
\end{pmatrix},\quad h_{1}=[h_{1,1}\; h_{1}'],\quad h_{i}=[0\; h_{i}'],\quad b'=\begin{pmatrix}b_{2}\\
\vdots\\
b_{d}
\end{pmatrix}.
\]
We have
\[
h\sigma_{b}=\begin{pmatrix}h_{1,1} & h_{1}'\\
0 & h_{2}'\\
\vdots & \vdots\\
0 & h_{d}'
\end{pmatrix}\begin{pmatrix}b_{1} & b'^T/2\\
b'/2 & \mathbf{0} 
\end{pmatrix}=\begin{pmatrix}h_{1,1}b_{1}+h_{1}'b'/2 & \hspace*{4mm} h_{1,1}b'^T/2\\
h_{2}'b'/2 & \mathbf{0} \\
\vdots & \vdots \\
h_{d}'b'/2 & \mathbf{0} 
\end{pmatrix},
\]
whence
\[
\begin{split}&h\sigma_{b}\,h^T  =\begin{pmatrix}h_{1,1}b_{1}+h_{1}'b'/2 & \hspace*{4mm} h_{1,1}b'^T/2\\
h_{2}'b'/2 & \mathbf{0}\\
\vdots & \vdots \\
h_{d}'b'/2 & \mathbf{0}
\end{pmatrix}\begin{pmatrix}h_{1,1} & 0 & \cdots & 0\\
h_{1}'^T & h_{2}'^T & \cdots & h_{d}'^T
\end{pmatrix}\\
 & =\begin{pmatrix}h_{1,1}(h_{1,1}b_{1}+h_{1}'b'/2)+h_{1,1}b'^Th_{1}'^T/2 & \hspace*{2mm} h_{1,1}b'^Th_{2}'^T/2 & \cdots & \hspace*{2mm} h_{1,1}b'^Th_{d}'^T/2\\
h_{1,1}h_{2}'b'/2\\
\vdots & & \mathbf{0} \\
h_{1,1}h_{d}'b'/2
\end{pmatrix}.
\end{split}
\]
Therefore, since $b'^Th_{i}'^T=h_{i}'b'$ for every $i$ and
$h_{i}'b'=h_{i}b$ for $i\ge2$, we obtain 
\[
\begin{split}
h\sigma_{b}\,h^T  &=h_{1,1}\begin{pmatrix}h_{1,1}b_{1}+h_{1}'b' & \hspace*{2mm}h_{2}'b'/2 & \cdots &\hspace*{2mm} h_{d}'b'/2\\
h_{2}'b'/2\\
\vdots& & \mathbf{0} \\
h_{d}'b'/2
\end{pmatrix}\\
  &=h_{1,1}\begin{pmatrix}h_{1}b_{1} & h_{2}b/2 & \cdots & h_{d}b/2\\
h_{2}b/2\\
\vdots & & \mathbf{0}\\
h_{d}b/2
\end{pmatrix},
\end{split}
\]
whence $h\sigma_{b}\,h^T = h_{1,1}\sigma_{hb}$, as desired.
\end{proof}
We use the notation
\[
g(\sigma,h)=\begin{pmatrix}h\\
\sigma h & h^{-T}
\end{pmatrix}\in Sp(d,\Bbb{R}),\qquad\sigma\in Sym(d,\Bbb{R}),h\in GL(d,\Bbb{R}).
\]
The product law is
\begin{equation}
g(\sigma_{1},h_{1})g(\sigma_{2},h_{2})=g(\sigma_{1}+\,h_{1}^{-T}\sigma_{2}h_{1}^{-1},h_{1}h_{2}).\label{eq:product in E}
\end{equation}
In the following result we show that $G=\Bbb{R}^{d}\rtimes H$ is isomorphic to the subgroup
of $Sp(d,\Bbb{R})$ given by $\Sigma\rtimes\rho(H):=g(\Sigma,\rho(H))$. This proves the first part of Theorem~\ref{thm:embedding}.
\begin{proposition}
\label{prop:phi}Take $H< T(d,\Bbb{R})_{+}$. Then the map
\[
\phi\colon\Bbb{R}_{d}\rtimes H\to g(\Sigma,\rho(H))<Sp(d,\Bbb{R}),\quad(b,h)\mapsto g(\sigma_{b},\rho(h))
\]
is a group isomorphism.
\end{proposition}
It is worth mentioning that Lemma 2.3 in \cite{2014-differentfaces} immediately follows from
this result.
\begin{proof}Recall that the product in $\Bbb{R}^{d}\rtimes H$
is defined by
\[
(b_{1},h_{1})(b_{2},h_{2})=(b_{1}+h_{1}b_{2},h_{1}h_{2}),\quad b_{i}\in\Bbb{R}^{d},h_{i}\in H.
\]
By definition of $\phi$ and using (\ref{eq:product in E}) there holds
\begin{eqnarray}\phi(b_{1},h_{1})\phi(b_{2},h_{2}) & =&g(\sigma_{b_{1}},\rho(h_{1}))g(\sigma_{b_{2}},\rho(h_{2}))\nonumber\\
 & =&g(\sigma_{b_{1}}+\,\rho(h_{1})^{-T}\sigma_{b_{2}}\rho(h_{1})^{-1},\rho(h_{1})\rho(h_{2})).\nonumber
\end{eqnarray}
Therefore, Lemma~\ref{lem:rho} gives
\begin{eqnarray}\phi(b_{1},h_{1})\phi(b_{2},h_{2}) & =&g(\sigma_{b_{1}}+\sigma_{h_{1}b_{2}},\rho(h_{1}h_{2}))\nonumber\\
 & =&g(\sigma_{b_{1}+h_{1}b_{2}},\rho(h_{1}h_{2}))\nonumber\\
 & =&\phi((b_{1},h_{1})(b_{2},h_{2})),\nonumber
\end{eqnarray}
as desired. Note that the fact that $g(\Sigma,\rho(H))$ is a subgroup
follows a posteriori.
\end{proof}
Intertwining the quasi-regular representation $\pi$, given in (\ref{eqn:def_quasireg}), with the Fourier transform $\mathcal{F}\colon\mathcal{H}\to L^{2}(\Theta_{L})$ we obtain the representation $\hat{\pi}(b,h):=\mathcal{F}\pi(b,h)\mathcal{F}^{-1}$
on $L^{2}(\Theta_{L})$ given by
\[
\hat{\pi}(b,h)\hat{f}(\xi)=|\det h\,|^{1/2}e^{-2\pi i\langle b,\xi\rangle}\hat{f}(h^T\xi),\qquad\hat{f}\in L^{2}(\Theta_{L}).
\]
The metaplectic representation restricted to $\Sigma\rtimes\rho(H)$ takes the form
\begin{equation}
\mu(\phi(b,h))\hat{f}(\xi)=|\det\rho(h)\,|^{-1/2}e^{\pi i\langle\sigma_{b}\xi,\xi\rangle}\hat{f}(\rho(h)^{-1}\xi),\qquad\hat{f}\in L^{2}(\Theta_{L}).\label{eq:meta}
\end{equation}
We now show that $\hat{\pi}$ and $\mu$ are unitarily equivalent, which concludes the proof of Theorem~\ref{thm:embedding}. The intertwining operator is given by
\[
\Psi\colon L^{2}(\Theta_{L})\to L^{2}(\Theta_{L}),\quad\Psi\hat{f}(\xi)=|\det J_{Q^{-1}}(\xi)|^{1/2}\hat{f}(Q^{-1}(\xi)),
\]
where $Q\colon\Theta_{L}\to\Theta_{L}$ is defined by $Q(\xi)=-\frac{1}{2}\xi_{1}\xi$.
\begin{proposition}
Let $\phi$ be the group isomorphism given by Proposition~\ref{prop:phi}.
For every $(b,h)\in\Bbb{R}^{d}\rtimes H$ there holds
\[
\Psi\mu(\phi(b,h))\Psi^{-1}=\mathcal{F}\pi(b,h)\mathcal{F}^{-1}=\hat{\pi}(b,h).
\]
\end{proposition}
\begin{proof}
We start by giving a few identities without proof \cite{2014-differentfaces}:
\begin{eqnarray}
 && |\det J_{Q}(\xi)|=2^{1-d}|\xi_{1}|^{d},\label{eq:det J_Q}\\
 && |\det J_{Q^{-1}}(\xi)|=2^{\frac{d}{2}-1}|\xi_{1}|^{-\frac{d}{2}},\label{eq:det J_Q^-1}\\
 && \langle\sigma_{b}\xi,\xi\rangle=-2\langle b,Q(\xi)\rangle,\label{eq:sigma Q}\\
 && Q^{-1}(\xi)=\sqrt{2}\xi/\sqrt{-\xi_{1}}.\label{eq:Q^-1}
\end{eqnarray}
By (\ref{eq:sigma Q}) and (\ref{eq:rho}) there holds
\[
-2\langle b,Q(h^T\xi)\rangle=\langle\sigma_{b}\,h^T\xi,h^T\xi\rangle=\langle h\sigma_{b}\,h^T\xi,\xi\rangle=h_{1,1}\langle\sigma_{hb}\xi,\xi\rangle.
\]
Therefore, using again (\ref{eq:sigma Q}) we obtain
\[
-2\langle b,Q(h^T\xi)\rangle=-2h_{1,1}\langle hb,Q(\xi)\rangle=-2\langle b,h_{1,1}\,h^TQ(\xi)\rangle,
\]
whence
\begin{equation}
Q(h^T\xi)=h_{1,1}\,h^TQ(\xi).\label{eq:h Q}
\end{equation}
By using the definition of $\Psi$, (\ref{eq:meta}), (\ref{eq:sigma Q}), (\ref{eq:def rho}) and once again the definition of $\Psi$, we can now compute for $\hat{f}\in L^{2}(\Theta_{L})$, $b\in\Bbb{R}^{d}$ and $h\in H$ 
\begin{eqnarray}\label{eq:almost there 0}
&&\Psi\mu  (\phi(b,h))\Psi^{-1}\hat{f}(\xi)=|\det J_{Q^{-1}}(\xi)|^{1/2}(\mu(\phi(b,h))\Psi^{-1}\hat{f})(Q^{-1}(\xi))\nonumber\\
 && =|\det J_{Q^{-1}}(\xi)|^{1/2}|\det\rho(h)|^{-1/2}e^{\pi i\langle\sigma_{b}Q^{-1}(\xi),Q^{-1}(\xi)\rangle}\Psi^{-1}\hat{f}(\rho(h)^{-1}Q^{-1}(\xi))\nonumber\\
 && =|\det J_{Q^{-1}}(\xi)|^{1/2}|\det\rho(h)|^{-1/2}e^{-2\pi i\langle b,\xi\rangle}\Psi^{-1}\hat{f}(\rho(h)^{-1}Q^{-1}(\xi))\\
 && =|\det J_{Q^{-1}}(\xi)|^{1/2}h_{1,1}^{-\frac{d}{4}}|\det h|^{1/2}e^{-2\pi i\langle b,\xi\rangle}\Psi^{-1}\hat{f}(h_{1,1}^{-\frac{1}{2}}\,h^TQ^{-1}(\xi))\nonumber\\
 && =|\det J_{Q^{-1}}(\xi)|^{1/2}h_{1,1}^{-\frac{d}{4}}|\det h|^{1/2}e^{-2\pi i\langle b,\xi\rangle}\nonumber\\
&&\hspace{4.5cm}\cdot|\det J_{Q}(h_{1,1}^{-\frac{1}{2}}\,h^TQ^{-1}(\xi))|^{1/2}\hat{f}(Q(h_{1,1}^{-\frac{1}{2}}\,h^TQ^{-1}(\xi))).\nonumber
\end{eqnarray}
Now note that by (\ref{eq:h Q}) and by the fact that $Q$ is quadratic there holds
\begin{equation}
Q(h_{1,1}^{-\frac{1}{2}}\,h^TQ^{-1}(\xi))=h_{1,1}\,h^TQ(h_{1,1}^{-\frac{1}{2}}Q^{-1}(\xi))=\,h^TQ(Q^{-1}(\xi))=\,h^T\xi.\label{eq:almost there}
\end{equation}
Moreover we have
\begin{eqnarray}
|\det J_{Q}(h_{1,1}^{-\frac{1}{2}}\,h^TQ^{-1}(\xi))| & =&2^{1-d}|(h_{1,1}^{-\frac{1}{2}}\,h^TQ^{-1}(\xi))_{1}|^{d}\nonumber\\
 & =&2^{1-d}h_{1,1}^{-\frac{d}{2}}|(\,h^TQ^{-1}(\xi))_{1}|^{d}\nonumber\\
 & =&2^{1-d}h_{1,1}^{-\frac{d}{2}}h_{1,1}^{d}|Q^{-1}(\xi)_{1}|^{d}\nonumber\\
 & =&2^{1-d}h_{1,1}^{\frac{d}{2}}2^{\frac{d}{2}}|\xi_{1}|^{\frac{d}{2}},\nonumber
\end{eqnarray}
where the first equality follows from (\ref{eq:det J_Q}), the third one from the fact that $h^T$ is lower triangular and the forth one from (\ref{eq:Q^-1}). Therefore by (\ref{eq:det J_Q^-1})
\begin{equation}
|\det J_{Q^{-1}}(\xi)|^{1/2}|\det J_{Q}(h_{1,1}^{-\frac{1}{2}}\,h^TQ^{-1}(\xi))|^{1/2}=2^{\frac{d}{4}-\frac{1}{2}}|\xi_{1}|^{-\frac{d}{4}}2^{\frac{1}{2}-\frac{d}{2}}h_{1,1}^{\frac{d}{4}}2^{\frac{d}{4}}|\xi_{1}|^{\frac{d}{4}}=h_{1,1}^{\frac{d}{4}}.\label{eq:almost there 2}
\end{equation}

Finally, inserting (\ref{eq:almost there}) and (\ref{eq:almost there 2}) into (\ref{eq:almost there 0}) we obtain
\[
\Psi\mu(\phi(b,h))\Psi^{-1}\hat{f}(\xi)  =|\det h|^{1/2}e^{-2\pi i\langle b,\xi\rangle}\hat{f}(\,h^T\xi) =\hat{\pi}(b,h)\hat{f}(\xi),
\]
as desired.
\end{proof}

\noindent\textbf{Acknowledgements} G.~S.\ Alberti was partially supported by the ERC Advanced Grant Project
MULTIMOD-267184. S. Dahlke was supported by Deutsche \linebreak Forschungsgemeinschaft (DFG), Grant DA 360/19--1. He also acknowledges the support of the Hausdorff Research Institute for Mathematics during the Special Trimester ``Mathematics of Signal Processing''. F. De Mari and E.~De Vito were partially supported by Progetto PRIN 2010-2011 ``Variet\`a reali e complesse: geometria, topologia e analisi armonica''. They are members of the Gruppo Nazionale per l'Analisi Matematica, la Probabilit\`a e le loro Applicazioni (GNAMPA) of the Istituto Nazionale di Alta Matematica (INdAM). H. F\"uhr acknowledges support from DFG through the grant Fu 402/5-1. 

Part of the work on this paper was carried out during visits of S. Dahlke and H. F\"uhr to Genova, and they thank the Universit\`a di Genova for its hospitality.

\end{document}